\documentclass[11pt]{amsart}

\usepackage{color}
\usepackage[TS1,T1]{fontenc}
\usepackage[latin1]{inputenc}
\usepackage{amsfonts,amssymb,amsmath,amsthm,enumerate,latexsym,comment}

\def\S{{\mathbb {S}}}
\def\R{{\mathbb {R}}}
\def\N{{\mathbb {N}}}

\def\F{{\mathcal {F}}}
\def\A{{\mathcal {A}}}

\newcommand{\p}{\partial}
\def\eps{{\varepsilon}}
\newtheorem{teo}{Theorem}[section]
\newtheorem{lema}[teo]{Lemma}
\newtheorem{prop}[teo]{Proposition}

\theoremstyle{remark}

\theoremstyle{definition}
\newtheorem{defi}[teo]{Definition}

\numberwithin{equation}{section}

\begin{document}

\title[Equations involving the $p(x)$-Laplacian with critical exponent in $\R^N$]
{Local existence conditions for an equations involving the $p(x)$-Laplacian with critical exponent in $\R^N$.}

\author[N. Saintier and A. Silva]{Nicolas Saintier and Analia Silva}

\address[N. Saintier]{Departamento de Matem\'atica, FCEyN - Universidad de Buenos Aires, Ciudad Universitaria, Pabell\'on I  (1428) Buenos Aires, Argentina.}

\address[A. Silva]{Instituto de Matem\'atica Aplicada San Luis, IMASL. Universidad Nacional de San Luis and CONICET. Ejercito de los Andes 950.
                   D5700HHW San Luis. Argentina.}

\email[A. Silva]{asilva@dm.uba.ar, acsilva@unsl.edu.ar}

\email[N. Saintier]{nsaintie@dm.uba.ar}

\urladdr[N. Saintier]{http://mate.dm.uba.ar/~nsaintie}



\begin{abstract}
The purpose of this paper is to formulate sufficient existence conditions for a critical  equation involving the $p(x)$-Laplacian of the form \eqref{MainEqu} below posed in $\R^N$. 
This equation is critical in the sense that the source term has the form $K(x)|u|^{q(x)-2}u$ with an exponent $q$ that can be equal to the critical exponent $p^*$ at some points of $\R^N$ including at infinity. 
The sufficient existence condition we find are local in the sense that they depend only on the behaviour of the exponents $p$ and $q$ near these points. 
We stress that we do not assume any symmetry or periodicity of the coefficients of the equation and that $K$ is not required to vanish in some sense at infinity like in most existing results. 
The proof of these local existence conditions is based on a notion of localized best Sobolev constant at infinity and 
a refined concentration-compactness at infinity. 
\end{abstract}

\maketitle

In this paper we address the existence problem for the $p(x)$-Laplace operator with a source that has critical growth in the sense of the Sobolev embeddings. To be precise, we consider the equation
\begin{equation}\label{MainEqu}
  -\Delta_{p(x)}u +k(x)|u|^{p(x)-2}u = f(x,u)  \qquad \text{ in }\R^N,
\end{equation}
where 
the $p(x)$-Laplacian operator $-\Delta_{p(x)}$ is defined as usual as $-\Delta_{p(x)}u=-\text{div}(|\nabla u|^{p(x)-2}\nabla u)$
and the source term $f$ has the form $f(x,u)=K(x)|u|^{q(x)-2}u$ for some nonnegative continuous function $K$ that has a  limit 
$K(\infty)$ at infinity. 
The exponents $p,q:\R^N\to [1,+\infty)$ are Log-H\"older continuous functions having a limit $p(\infty)$ and $q(\infty)$ at  infinity and satisfying 
$1<\inf_{\R^N} p\le \sup_{\R^N} p<N $ and $1\le q(x)\le p^*(x):=Np(x)/(N-p(x))$, $x\in\R^N$. 
We will assume that equation \eqref{MainEqu} is critical at infinity in the sense that $q(\infty)=p^*(\infty)$. 
Notice that the critical set $\A:=\{x\in\R^N:\,q(x)=p^*(x) \}$ can be non-empty as well. 
The main purpose of this paper is to find conditions on the coefficients $p,q,k,K$ to obtain the existence of a non-trivial solution to  \eqref{MainEqu} without any periodicity or symmetry assumptions of these coefficients and without 
requiring $K$ to vanish  at infinity.  

\bigskip

The $p(x)$-Laplacian, and more generally variable exponent spaces, have been the subject of an intense research activity
 in the fields of both partial differential equations and harmonic analysis.
This interest comes from the possible applications of these spaces as well as the new variety of phenomenon that appear in comparison to the traditional setting of constant exponent spaces.

From the point of view of applications, the $p(x)$-Laplacian appears in two
main applied problems. The first one deals with fluid mechanic where 
this operator is useful in the modelling of the so-called
electrorheological fluids. These fluids have the peculiarity of
modifying their mechanical properties in presence of external
factors such as electromagnetic field. We refer to M.
Ru{\v{z}}i{\v{c}}ka's book \cite{Ru} where this theory is fully developed.
The second application comes form the field of image processing
where the $p(x)$-Laplacian is used to design image restauration
processes that behave differently according to the smoothness of the
image. This way noise  can be removed from an image preserving the
boundaries. We refer to Y. Chen, S. Levin and R. Rao's paper
\cite{CLR}.

From the points of view of partial differential equations, we
mention as an example one striking features of the variable exponent
setting and refer to \cite{HHLN} and \cite{Radulescu} for general
overviews of recent results concerning PDE involving the
$p(x)$-Laplacian. It is well-known that, in the constant exponent
setting, the immersion of the Sobolev space $W_0^{1,p}(U)$, with $U$ a
smooth bounded domain, into $L^{Np/(N-p)}(U)$ is never compact.
However in the variable setting  the immersion of
$W^{1,p(\cdot)}(U)$ into $L^{q(\cdot)}(U)$ can be compact even if
the critical  set $\{x\in \bar U: \, q(x)=p^*(x)\}$ is non-empty
provided  that  this set is "`small"' and we have a control on how
the exponent $q$ reaches $p^*$ (see \cite{MOSS}). Moreover it was recently proved in \cite{FBSS1} that
there exists an extremal for the embedding of $W_0^{1,p(\cdot)}(U)$
into $L^{q(\cdot)}(U)$ with $q(\cdot)\le p^*(\cdot)$, if $q$ is subcritical in a sufficiently big
ball.

\bigskip

In the constant exponent setting, equations like \eqref{MainEqu}  in a bounded domain with critical exponent, or in an unbounded domain with subcritical or critical exponent, have been widely studied since the seminal paper by Aubin \cite{Aubin} and Brezis-Nirenberg  \cite{BN} who dealt with equations involving the Laplacian operator. 
The initial motivation for the study of this kind of equations comes from their appearance in differential geometry (the so-called Yamab\'e problem). 
Aubin's and Brezis-Nirenberg's results and methods have been extended in various directions to deal with critical equations 
involving  the $1$-Laplacian, $p$-Laplacian, the bilaplacian, ... with various boundary conditions, either in domains of $\R^N$ or on Riemannian manifolds.

In contrast  there are relatively few results in the variable exponent setting.
Concerning the subcritical case, we mention the papers \cite{FanHan}, \cite{Alves2}, \cite{Fan2},  \cite{Fan3}, \cite{FuZhang},  \cite{AF2}.
For the critical case, we refer to \cite{Also}, \cite{Alves3}, \cite {Fu2}, \cite{FuZhang2}, \cite{LiangZhang}, \cite{AF1},
\cite{AF3}.
 The authors of these papers prove existence and multiplicity of  solutions for equation similar to  \eqref{MainEqu}.
The major difficulty in proving the existence of solutions is the
possibility of a loss of mass at infinity and, in the case of
critical problem, the possibility of concentration, either around
some points or at infinity. The authors of the aformentionned papers
circumvent this problem usually assuming one or various of the
following conditions: (i) some symmetry on the coefficients of the
equation (in most cases the coefficients are supposed radial) since
it is well-known since Strauss' paper \cite{Strauss} (see also
\cite{Lions2})  that the presence of symmetry improves Sobolev
embeddings, (ii) some periodicity condition since the periodicity
virtually reduces the problem to a problem set in an bounded domain, 
(iii) perturbing $f$ adding some subcritical term, (iv) requiring
that $K$ vanishes at infinity in the sense that $K$ belongs to some
$L^{r(\cdot)}(\R^N)$ space or that $\lim_{|x|\to +\infty}K(x)=0$.

\bigskip

To deal with the concentration phenomenon, the famous
concentration-compactness principle (CCP) due to Lions \cite{Lions},
originally formulated for critical problems in bounded domain and
later extended to deal with critical problem in unbounded domains by
Chabrowski \cite{Chabrowski}, is of prime importance since it
describes the concentration by a weighted sum of Dirac masses and
the loss of mass by measures "`supported at infinity"`. The CCP
allows to  formulate existence conditions for critical equations
that are local in the sense that they rely only on the behaviour of
the coefficients of the equation near the possible concentration
points.

The CCP  has recently been extended to the variable exponent setting in \cite{FBS1} (and independently in \cite{Fu})  for a bounded domain,
and in \cite{Fu2} and \cite{FuZhang2} for an unbounded domain.
However in these versions, the Sobolev constant used in the statement to compare the weights of the Dirac masses is a global Sobolev constant which does not reflects the local behaviour of the exponents $p$ and $q$ near the concentrations points. A refined version of the CCP for bounded domain was then later proved in \cite{FBSS1} introducing the notion of localized Sobolev constant. This notion was then used in \cite{FBSS2} to formulate existence conditions for an equation like
\ref{MainEqu} in a bounded domain relying only  on the local behaviour near concentration points.
We also mention the papers \cite{FBSS3} and \cite{FBSS4} where the same program as been carried for equations on a bounded domain with nonlinear boundary conditions.
A systematic and self-contained treatment of critical Sobolev immersion and critical equations in bounded domain can be found in the book \cite{FBSS5}. 

\bigskip

Our  purpose in this paper is twofold. First we want to prove a
refined CCP at infinity in the spirit of \cite{Chabrowski} together
with a notion of localized Sobolev constant at infinity so as to
capture the behaviour of $p$ and $q$ at infinity. We then want to
apply this CCP to provide sufficient local existence conditions for
the equation \ref{MainEqu}  without any of the assumptions usually
done. In particular we will not require any periodicity or symmetry
of the coefficients nor will we require  the function $K$ to vanish
in some sense at infinity. Instead we will see that, analogously to
what we did in \cite{FBSS1}, \cite{FBSS2}, \cite{FBSS3},
\cite{FBSS4}, if $K$ does not vanish at infinity,  then, to prevent
the loss of mass or concentration at infinity, it is enough to
require that $p$ and $q$ have a local minimum, resp. maximum, at
infinity, and that the hessian matrix of $p$ or $q$ at infinity (see
definition \ref{DefiC2Infinity} below) is non-zero.

\bigskip

The paper is organized as follows. In the first section we collect
some basic definitions and results concerning the variable exponent
Lebesgue and Sobolev spaces that will be needed in this paper. We
state the main assumptions on the exponents $p$ and $q$ in the second section. We
then introduce the notion of localized best Sobolev constant at
infinity in the next section and study its relation with the usual
Sobolev constant. As an application we prove in the fourth section a
refined version of the concentration-compactness principle at
infinity that we apply in the next section to obtain a sufficient
condition of existence for equation \ref{MainEqu}. We collect in 
appendix A a result concerning a compact embedding and in Appendix B
the lengthy test-function computations  used to prove the
local existence conditions.

\section{Preliminaries on variable exponent spaces.}

In this section we review some preliminary results regarding Lebesgue and Sobolev spaces with variable exponents that 
will be used throughout this paper. 
All  these results and a comprehensive study of these spaces with an exhaustive bibliography can be found in \cite{libro}.

\subsection{Lebesgue spaces}

Consider a measurable function $p:\R^N \to [1,+\infty]$. The
variable exponent Lebesgue space $L^{p(\cdot)}(\R^N)$ is defined by
$$ L^{p(\cdot)}(\R^N) = \Big\{u\in L^1_{\text{loc}}(\R^N) \colon \int_{\R^N} |u(x)|^{p(x)}\,dx<\infty\Big\}. $$
This space is endowed with the norm
$$ \|u\|_{L^{p(\cdot)}(\R^N)}=\inf\Big\{\lambda>0:\int_ {\R^N} \Big|\frac{u(x)}{\lambda}\Big|^{p(x)}\,dx\leq 1\Big\}. $$
The following H\"older inequality holds (see e.g. \cite{libro}[lemma
3.2.20]):

\begin{prop} Let $p,q,s:\R^N\to [1,+\infty] $ be measurable functions such that
$$ \frac{1}{s(x)} = \frac{1}{p(x)} + \frac{1}{q(x)} \qquad \text{for a.e. $x\in\R^N$.} $$
Let $f\in L^{p(\cdot)}(\R^N)$ and $g\in L^{q(\cdot)}(\R^N)$. Then
$fg\in L^{s(\cdot)}(\R^N)$ with
\begin{equation}\label{HolderIneq}
  \|fg\|_{L^{s(\cdot)}(\R^N)}   \le \Big(\Big(\frac{s}{p}\Big)^++\Big(\frac{s}{q}\Big)^+\Big)
    \|f\|_{L^{p(\cdot)}(\R^N)}     \|g\|_{L^{q(\cdot)}(\R^N)}. 
\end{equation}
\end{prop}

It is usually convenient to study the so-called modular
$\rho(u):=\int_{\R^N} |u|^{p(x)}\,dx$ instead of the norm
$\|u\|_{L^{p(\cdot)}(\R^N)} $. The following result provide some relations
between the two (see \cite{libro}[Chap 2-1]):

\begin{prop}\label{norma.y.rho}
Assume that $p$ is bounded. For $u\in L^{p(\cdot)}(\R^N)$ and
$\{u_k\}_{k\in\N}\subset L^{p(\cdot)}(\R^N)$, we have
\begin{align}
& u\neq 0 \Rightarrow \Big(\|u\|_{L^{p(\cdot)}(\R^N)} = \lambda \Leftrightarrow \rho(\frac{u}{\lambda})=1\Big).\nonumber \\
& \|u\|_{L^{p(\cdot)}(\R^N)}<1 (=1; >1) \Leftrightarrow \rho(u)<1(=1;>1).\nonumber  \\
& \|u\|_{L^{p(\cdot)}(\R^N)}>1 \Rightarrow \|u\|^{p^-}_{L^{p(\cdot)}(\R^N)} \leq \rho(u) 
\leq \|u\|^{p^+}_{L^{p(\cdot)}(\R^N)}.\nonumber \\
& \|u\|_{L^{p(\cdot)}(\R^N)}<1 \Rightarrow \|u\|^{p^+}_{L^{p(\cdot)}(\R^N)} 
\leq \rho(u) \leq \|u\|^{p^-}_{L^{p(\cdot)}(\R^N)}.\nonumber \\
& \lim_{k\to\infty}\|u_k\|_{L^{p(\cdot)}(\R^N)} = 0 \Leftrightarrow \lim_{k\to\infty}\rho(u_k)=0.\nonumber \\
& \lim_{k\to\infty}\|u_k\|_{L^{p(\cdot)}(\R^N)} = \infty
\Leftrightarrow \lim_{k\to\infty}\rho(u_k) = \infty.\nonumber
\end{align}
\end{prop}

Assume now that $1<p^- \le p^+ <\infty$, where $p^-:=ess-\inf_{\R^N}
p$ and $p^+:=ess-\sup_{\R^N} p$. Then $L^{p(\cdot)}({\R^N})$ is a
separable and reflexive Banach space, and the smooth functions with compact
support are dense in $L^{p(\cdot)}(\R^N)$.

\subsection{Sobolev spaces}

Consider a measurable function $p:\R^N \to [1,+\infty]$. The
variable exponent Sobolev space $W^{1,p(\cdot)}(\R^N)$ is defined by
$$ W^{1,p(\cdot)}(\R^N) = \{u\in W^{1,1}_{\text{loc}}(\R^N) \colon u\in L^{p(\cdot)}(\R^N) \mbox{ and } |\nabla u |\in  L^{p(\cdot)}(\R^N)\}. $$
The corresponding norm for this space is
$$ \|u\|_{W^{1,p(\cdot)}(\R^N)}=\|u\|_{L^{p(\cdot)}(\R^N)}+\| \nabla u \|_{L^{p(\cdot)}(\R^N)}. $$
As with the variable Lebesgue spaces, if $1<p^- \le p^+ <\infty$
then $W^{1,p(\cdot)}(\R^N)$ is a separable and reflexive Banach space and 
 the smooth functions  with compact support are dense in $W^{1,p(\cdot)}(\R^N)$.

As usual, we denote the conjugate exponent of $p(.)$ by $p'(x) = p(x)/(p(x)-1)$  and the Sobolev exponent by
$$ p^*(x)=\begin{cases} \frac{Np(x)}{N-p(x)} & \mbox{ if } p(x)<N,\\ \infty & \mbox{ if } p(x)\geq N. \end{cases} $$
The standard Sobolev embedding theorem and Poincar\'e inequality still hold provided that $p$ satisfies a regularity condition called Log-H\"older continuity (see \cite{libro}[Chap. 4-1]) defined by the following condition: there exists $C>0$ such that 
$$ |p(x)-p(y)|\le \frac{C}{\ln(e+1/|x-y|)},\qquad x,y\in\R^N, \, x\neq y.  $$

\begin{prop}
Let $p:\R^N\to [1,N)$ be a Log-H\"older continuous function and $q:\R^N\to [1,+\infty)$ be a  measurable function 
 such that $q\le p^*$. Then, for any bounded smooth domain $U\subset \R^N$, $\|u\|:=\|\nabla u\|_{L^{p(\cdot)}(U)}$ is a 
norm in $W_0^{1,p(\cdot)}(U)$ equivalent to the usual norm. Moreover there is a continuous embedding from
$ W^{1,p(\cdot)}(U)$ into $L^{q(\cdot)}(U)$, which is also compact if  $q$ is strictly subcritical in 
the sense that $\inf_U p^*-q>0$. 
\end{prop}

\section{Assumptions on the exponents $p$ and $q$}

In this paper, the exponents $p$ and $q$ will always be measurable functions $p,q:\R^N\to [1,+\infty]$ satisfying the following assumptions: 

\begin{enumerate}
\item[(H1)] $p$ and $q$ have a modulus of continuity $\rho$ in the sense that for any $x,h\in\R^N $,
$$ p(x+h)=p(x)+\rho(h),\qquad \text{and} \qquad q(x+h)=q(x)+\rho(h) $$ 
with $\lim_{h\to 0}\rho(h)\ln\,|h|=0$.
\end{enumerate}
\begin{enumerate}
\item[(H2)] there exist real numbers $p(\infty)$ and $q(\infty)$ such that
\begin{equation}\label{Decay}
  \lim_{|x|\to +\infty} |p(x)-p(\infty)|\ln\,|x|=0 \qquad \text{and}\qquad
        \lim_{|x|\to +\infty} |q(x)-q(\infty)|\ln\,|x|=0
\end{equation}
\end{enumerate}
\begin{enumerate}
\item[(H3)] denoting $p^-:=\inf_{\R^N} p$, $p^+:=\sup_{\R^N} p$, we assume that
$$ 1<p^-\le p^+<N \qquad \text{and}\qquad 1\le q\le p^*:=\frac{Np}{N-p}. $$
We denote $\A:=\{x\in\R^N:\, q(x)=p^*(x)\}$ the critical set which can be empty or not. \\
We  assume that $\infty$ is critical in the sense that  $q(\infty)=p^*(\infty)$.
\end{enumerate}

Assumption (H1) is slightly stronger than the usual Log-H\"older continuity assumption, which is natural to assume to have 
the Sobolev embeddings theorems. We need this stronger version to be able to perform some test-function computation. 
Assumption (H2) can be thought in the same way as (H1) as a stronger version of a Log-H\"older continuity at infinity. 
Log-H\"older continuity at infinity (also called Log-H\"older decay condition) is a classical assumption when dealing for instance with the maximal function in $\R^N$ (see \cite{libro}[Chap.4]). 
Assumption (H3)  says that we are assuming that the infinity is critical. We stress that the critical set $\A$ can contain other  points that $\infty$ where $q=p^*$. The case where infinity is subcritical, in the sense that $ q(\infty)<p^*(\infty)$, will be treated elsewhere.

\section{Localized best Sobolev constant at infinity}

Given a smooth open subset $U\subset \R$ we denote by $S(p(\cdot),q(\cdot),U)$ the best Sobolev constant for the embedding of
$W_0^{1,p(\cdot)}(U)$ into $L^{q(\cdot)} (U)$, namely
$$ S(p(\cdot),q(\cdot),U) := \inf_{u\in W_0^{1,p(\cdot)}(U),\, u\neq 0} 
\frac{\|\nabla u\|_{L^{p(\cdot)}(U)}}{\|u\|_{L^{q(\cdot)}(U)}}. $$
Let $x_0\in\A$ be a critical point i.e. $q(x_0)=p^*(x_0)$. 
Taking in the previous definition a  ball  $U=B_{x_0}(\eps)$ centered at $x_0$ with small radius $\eps>0$, 
and noticing that $S(p(\cdot),q(\cdot),B_{x_0}(\eps))$ is non-decreasing in $\eps$,
we can consider the best localized Sobolev constant $\bar S_{x_0}$ at $x_0$ defined by
\begin{equation}\label{DefLocCste}
  S_{x_0} = \lim_{\eps\to 0} S(p(\cdot),q(\cdot),B_{x_0}(\eps)) = \sup_{\eps>0}S(p(\cdot),q(\cdot),B_{x_0}(\eps)). 
\end{equation}
This notion was introduced in \cite{FBSS1} to study precisely the concentration phenomenon at $x_0$ and obtain a refined concentration-compactness principle.

To study the Sobolev embedding at infinity, we consider  in an analogous way, the best Sobolev constant
$S_R=S(p(\cdot),q(\cdot),\R^N\setminus B_R) $ in $\R^N\backslash B_R$, where we denote $B_R:=B_0(R)$, $R>0$.
Noticing that $S_R$ is non-decreasing in $R$, we define the localized best Sobolev constant $S_\infty$  at infinity taking the limit of $S_R$ as $R\to +\infty$:

\begin{defi}
The localized best Sobolev constant $S_\infty$  at infinity is defined as
\begin{equation}\label{DefSinf}
 S_\infty = \lim_{R\to +\infty} S(p(\cdot),q(\cdot),\R^N\setminus B_R) = 
  \sup_{R>0} S(p(\cdot),q(\cdot),\R^N\setminus B_R).
\end{equation}
\end{defi}

Since the localized best Sobolev constant $ S_{x_0}$ depends only on the behaviour of $p$ and $q$ near a critical point $x_0$,
it is natural to try to compare it with the usual best Sobolev constant $K(N,r)$, $r\in [1,N)$, corresponding to the embedding of the constant-exponent space
$D^{1,r}(\R^N)$, the closure of $C^\infty_c(\R^N)$ for the norm $\|\nabla u\|_{L^r(\R^N)}$, into $L^{Nr/(N-r)}(\R^N)$, namely
\begin{equation}\label{BestSobCste}
 K(N,r)^{-r} := \inf_{u\in C^\infty_c(\R^N),\, u\neq 0} \frac{\displaystyle \int_{\R^N} |\nabla u|^r\,dx }
                 {\displaystyle \Big(\int_{\R^N} |u|^\frac{Nr}{N-r}\,dx\Big)^\frac{N-r}{N} }.
\end{equation}
It was proved in \cite{FBSS1} that $S_{x_0}\le K(N,p(x_0))^{-1}$ with equality if $p$ has a local minimum at $x_0$ and $q$ a local maximum.
The next two results show that these two properties still holds for $S_\infty$.

\begin{prop}\label{PropIneq}
Assume  that $\infty$ is critical in the sense that $q(\infty)=p(\infty)^*$. Then
\begin{equation}\label{CotaLocSob}
 S_\infty \leq K(N,p(\infty))^{-1},
\end{equation}
where $K(N,p(\infty))$ is defined in \eqref{BestSobCste} (with $r=p(\infty)$) and $S_\infty$ in \eqref{DefSinf}.
\end{prop}

\begin{proof}
Given $\phi\in C^\infty_c(\R^N)$, $\phi\not\equiv 0$, we consider the rescaled function $\phi_\lambda$  defined for $\lambda>0$ small by
\begin{equation}\label{rescaled}
\phi_\lambda(x) =\lambda^\frac{-N}{p^*(x_\lambda)}\phi\Big(\frac{x-x_\lambda}{\lambda}\Big),
\end{equation}
where the points $x_\lambda\in \R^N$ are such that $|x_\lambda|\ge 1/\lambda$.
It then follows from  assumptions (H1) and (H2) that
\begin{equation}\label{cont}
 q(x_\lambda + \lambda y)=q(\infty)+\eps_\lambda^1(y), \qquad
p(x_\lambda + \lambda y)=p(\infty)+\eps_\lambda^2(y),
\end{equation}
and
\begin{equation}\label{contbis}
 p(x_\lambda + \lambda y)=p(x_\lambda)+\eps_\lambda^3(y),
\end{equation}
where the functions $\eps_\lambda^i$, $i=1,2,2$, satisfy
\begin{equation}\label{cont2}
 |\eps_\lambda^i(y)|\le \eps_\lambda \qquad \text{with} \qquad
\lim_{\lambda\to 0} \eps_\lambda\ln\,\lambda=0
\end{equation}
uniformly for  $y$ in a compact set.

Fix some $R>0$. Since $\lim_{\lambda\to +\infty}|x_\lambda|=+\infty$, we have that  $supp \,\phi_\lambda\subset \R^N\backslash B_R$ for $\lambda>0$ small enough.
It follows that
\begin{equation}\label{Eq1}
 S(p(\cdot),q(\cdot),\R^N\setminus B_R) \le \liminf_{\lambda\to 0}\frac{\|\nabla \phi_\lambda\|_{p(\cdot)}}{\|\phi_\lambda\|_{q(\cdot)}}.
\end{equation}
We now estimate each terms in the right hand side of this
inequality. First, in view of (\ref{cont}),
$$ \int_{\R^N} |\phi_\lambda(x)|^{q(x)}\,dx
 =  \displaystyle\int_{\R^N} \lambda^{N\Big(1-\frac{ (q(\infty)+\eps_\lambda^1(y)) }
 { (p(\infty)+\eps^2_\lambda(y))^* }\Big)}
|\phi(y)|^{q(\infty)+\eps_\lambda^1(y)}\,dy.
$$
Recalling that $\phi$ has compact support and that $q(\infty)=p(\infty)^*$, we deduce using  \eqref{cont2} that
$$ \lim_{\lambda\to +\infty}  \int_{\R^N} |\phi_\lambda(x)|^{q(x)}\,dx
=\int_{\R^N} |\phi|^{q(\infty)} \,dx. $$
According to the definition of the $\|\cdot\|_{q(\cdot)}$ norm we can then write
$$  1 = \int_{\R^N }   \Big(\frac{\phi_\lambda(x)}{\|\phi_\lambda\|_{q(\cdot)}}\Big)^{q(x)}\,dx
= (1+o(1)) \|\phi_\lambda\|_{q(\cdot)}^{-q(\infty)+o(1)}  \int_{\R^N} |\phi|^{q(\infty)}\,dx.
$$
We then deduce that
\begin{equation}\label{Eq2}
 \lim_{\lambda\to 0}  \|\phi_\lambda\|_{q(\cdot)} = \|\phi\|_{q(\infty)}.
\end{equation}

We treat the gradient term analogously. First, as $\lambda\to 0$,
\begin{eqnarray*}
 \int_{\R^N} |\nabla \phi_\lambda(x)|^{p(x)}\,dx
= \int_{\R^N} \lambda^{N\Big( 1-\frac{p(x_\lambda+\lambda y)}{p^*(x_\lambda)}\Big)- p(x_\lambda+\lambda y)} |\nabla\phi|^{p(x_\lambda)+\lambda y}\,dy
 \to \int_{\R^N} |\nabla\phi|^{p(\infty)}\,dy.
\end{eqnarray*}
Then, writing
$\int_{\R^N}\Big(\frac{|\nabla\phi_\lambda(x)|}{\|\nabla\phi_\lambda\|_{p(\cdot)}}\Big)^{p(x)}\,dx
=1$  as above, we deduce that
\begin{equation}\label{Eq3}
 \lim_{\lambda\to 0} \|\nabla \phi_\lambda\|_{p(\cdot)}=\|\nabla \phi\|_{p(\infty)}.
\end{equation}

Coming back to (\ref{Eq1}), we thus deduce in view of \eqref{Eq2} and \eqref{Eq3} that when $q(\infty)=p(\infty)^*$, there holds that
$$ S(p(\cdot),q(\cdot),\R^N \backslash B_R)\leq\frac{\|\nabla\phi\|_{p(\infty), \R^N}}{\|\phi\|_{p(\infty)^*, \R^N}} $$
for every $\phi\in C^\infty_c(\R^N)$ and any $R>0$. It follows that for any $R>0$,
$$ S(p(\cdot),q(\cdot),\R^N \backslash B_R)\leq K^{-1}(N,p(\infty)). $$
Taking the limit $R\to +\infty$ yields \eqref{CotaLocSob}.

\end{proof}

The next result gives a sufficient condition on $p$ and $q$ to have
the equality in \eqref{CotaLocSob}. The proof follows the line of
\cite{FBSS1} with an additional difficulty. Indeed the proof there
ultimately relies on H\"older inequality in a ball $B_{x_0}(\eps)$.
Here the analogous of this ball is the ball centered at infinity
$\R^N\backslash B_R$ which has infinite volume. To overcome this
difficulty, we adapt a trick apparently due to Nekvinda
\cite{Nekvinda} which amounts to say that, for a given
 constant $c\in (0,1)$ and a variable exponent $\eps(x)$, an integral like $\int_{\R^N\backslash B_R} c^{\eps(x)}\,dx $ can be made
arbitrarily small provided that $\eps$ converge fast enough to $+\infty$ at infinity.

\begin{teo}\label{PropIgaldad}
If $\infty$ is a local maximum of $q$ and a local minimum of $p$ in the sense that
\begin{equation}\label{HypEgal}
 p(\cdot)\ge p(\infty) \qquad \text{ and } \qquad q(\cdot)\le q(\infty)
\end{equation}
for $|x|$ large, where $p(\infty)$ and $q(\infty)$ are given in assumption (H2),
then, if $q(\infty)=p^*(\infty)$,
\begin{equation}\label{Igualdad}
S_\infty = K(N,p(\infty))^{-1},
\end{equation}
where $K(N,p(\infty))$ is defined in \eqref{BestSobCste} and $S_\infty$ in \eqref{DefSinf}.
\end{teo}

\begin{proof} Notice that the result is trivial if $p$ is constant in $\R^N\backslash B_R$ for some $R>0$.
We thus assume from now on that the set $\{x\in \R^N\backslash B_R\text{ s.t. } p(x)\neq p(\infty)\}$ has positive measure for any $R>0$ big.

In view of \eqref{CotaLocSob}, we only have to prove that
\begin{equation}\label{ToBeProvedEgal}
S_\infty \ge K^{-1}(N,p(\infty)).
\end{equation}

As a first step, we claim that for any $R>0$ and $u\in C^\infty_c(\R^N\backslash B_R)$,
\begin{equation}\label{CambioVariable1}
 \|u\|_{L^{q(\cdot)}(\R^n\backslash B_R)} = R^\frac{N}{q(\infty)}(1+o(1))\|u_R\|_{L^{q_R(\cdot)}(\R^N\backslash B_1)}
\end{equation}
and
\begin{equation}\label{CambioVariable2}
\|\nabla u\|_{L^{p(\cdot)}(\R^n\backslash B_R)} = R^\frac{N}{p(\infty)^*}(1+o(1))\|\nabla u_R\|_{L^{p_R(\cdot)}(\R^N\backslash B_1)},
\end{equation}
where $u_R(x):=u(Rx)$, $p_R(x):=p(Rx)$, $q_R(x):=q(Rx)$, and $o(1)$ does not depend on $u$ and goes to $0$ as $R\to \infty$ uniformly in $y\in \R^N\backslash B_R$.
We only prove \eqref{CambioVariable1} since the proof of  \eqref{CambioVariable2} is similar.
We have
\begin{eqnarray*}
 \|u\|_{L^{q(\cdot)}(\R^n\backslash B_R)}
   &=& \inf\,\Big\{ \lambda>0 \text{ s.t. } \int_{\R^n\backslash B_R} \Big|\frac{u(x)}{\lambda}\Big|^{q(x)}\,dx \le 1 \Big\} \\
     &=& \inf\,\Big\{ \lambda>0 \text{ s.t. } \int_{\R^n\backslash B_1} \Big|\frac{u_R(x)}{\lambda R^{-N/q_R(y)}}\Big|^{q_R(y)}\,dy\le 1\Big\}
\end{eqnarray*}
with
$$ R^{-\frac{N}{q_R(y)} + \frac{N}{q(\infty)}} = \text{exp}\Big\{ \frac{N}{q_R(y)q(\infty)}\Big( q(Ry)-q(\infty)\Big) \ln\,R\Big\}
$$
which goes to $1$ as $R\to\infty$ uniformly in $y\in \R^N\backslash B_R$ in view of \eqref{Decay} and recalling that $q$ is bounded.
It follows that
\begin{eqnarray*}
 \|u\|_{L^{q(\cdot)}(\R^n\backslash B_R)}
&=& \inf\,\Big\{ \lambda>0 \text{ s.t. } \int_{\R^n\backslash B_1} \Big|\frac{u_R(x)}{\lambda R^{-N/q(\infty)}(1+o(1))}\Big|^{q_R(y)}\,dy\le 1\Big\},
\end{eqnarray*}
from which we deduce \eqref{CambioVariable1}.

It follows from \eqref{CambioVariable1}, \eqref{CambioVariable2} and the fact that the map
$u\in C^\infty_c(\R^N\backslash B_R)\to u_R\in  C^\infty_c(\R^N\backslash B_1)$ is biyective, that when $q(\infty)=p(\infty)^*$,
$$ S(p(\cdot),q(\cdot),\R^N\backslash B_R) = (1+o(1))S(p_R(\cdot),q_R(\cdot),\R^N\backslash B_1) $$
where $o(1)\to 0$ as $R\to \infty$. To conclude the proof, it thus suffices to prove that
$$ \liminf_{R\to \infty}S(p_R(\cdot),q_R(\cdot),\R^N\backslash B_1)\ge S(p(\infty),q(\infty),\R^N\backslash B_1).  $$
This will hold if we can prove that for any $u\in  C^\infty_c(\R^N\backslash B_1)$,
\begin{equation}\label{Holder1}
 \|\nabla u\|_{L^{p(\infty)}(\R^N\backslash B_1)} \le (1+o(1)) \|\nabla u\|_{L^{p_R(\cdot)}(\R^N\backslash B_1)}
\end{equation}
and
\begin{equation}\label{Holder2}
 \|u\|_{L^{q(\infty)}(\R^N\backslash B_1)} \ge (1+o(1)) \|u\|_{L^{q_R(\cdot)}(\R^N\backslash B_1)}.
\end{equation}
We prove \eqref{Holder1} (the proof of \eqref{Holder2} is similar). Fix some $c\in (0,1)$.
Using H\"older inequality \eqref{HolderIneq} and the hypothesis that $\infty$ is a local minimum of $p$,
we have
\begin{eqnarray*}
 \|\nabla u\|_{L^{p(\infty)}(\R^N\backslash B_1)}
&=& c^{-1} \|c\nabla u\|_{L^{p(\infty)}(\R^N\backslash B_1)}  \\
&\le &  c^{-1}\Big(\Big(\frac{p(\infty)}{p_R}\Big)^++\Big(\frac{p(\infty)}{s_R}\Big)^+\Big)
          \|\nabla u\|_{L^{p_R(\cdot)}(\R^N\backslash B_1)} \|c\|_{L^{s_R(\cdot)}(\R^N\backslash B_1)},
\end{eqnarray*}
where $s_R$ is defined by $\frac{1}{s_R} = \frac{1}{p(\infty)} - \frac{1}{p_R}$ and
$\Big(\frac{p(\infty)}{p_R}\Big)^+ = \sup_{\R^N\backslash B_1}\frac{p(\infty)}{p_R}$.
It follows from (H2) that
\begin{eqnarray*}
 \|\nabla u\|_{L^{p(\infty)}(\R^N\backslash B_1)} \le c^{-1}(1+o(1))  \|\nabla u\|_{L^{p_R(\cdot)}(\R^N\backslash B_1)} \|c\|_{L^{s_R(\cdot)}(\R^N\backslash B_1)}.
\end{eqnarray*}
We now estimate $\|c\|_{L^{s_R(\cdot)}(\R^N\backslash B_1)}$. We write $c$ in the form $c=e^{-\gamma}$, $\gamma>0$. Then
\begin{eqnarray*}
 \int_{\R^N\backslash B_1} c^{s_R(x)}\,dx = \int_{\R^N\backslash B_1} \text{exp}\Big( -\gamma \frac{p(\infty)p_R(x)}{p_R(x)-p(\infty)}\Big) \,dx
 \le  \int_{\R^N\backslash B_1} \text{exp}\Big( - \frac{\gamma}{p_R(x)-p(\infty)}\Big) \,dx.
\end{eqnarray*}
The last inequality follows from the fact that $p_R\ge p(\infty) \ge 1$.
Fix some $\delta>0$ small such that $\frac{\gamma}{\delta}>N$.
Then for $R$ big (depending on $\delta$) we have in view of assumption \eqref{Decay} that
$$ 0\le p_R(x)-p(\infty)\le \frac{\delta}{\ln\,R|x|} \qquad \text{for } |x|\ge 1.$$
Then
$$ \int_{\R^N\backslash B_1} c^{s_R(x)}\,dx \le \int_{\R^N\backslash B_1} \Big(R|x|\Big)^{-\frac{\gamma}{\delta}}\,dx
 = R^{-\frac{\gamma}{\delta}}\omega_{N-1}\int_1^{+\infty} r^{N-1-\frac{\gamma}{\delta}}\,dr
= \frac{R^{-\frac{\gamma}{\delta}}\omega_{N-1}}{\frac{\gamma}{\delta}-N}
$$
which goes to $0$ as $R\to \infty$.
Recall that we assume that the set  $\{x\in \R^N\backslash B_1\text{ s.t. } s_R(x)\neq \infty)\}$ has positive measure for any $R>0$ big.
It follows that $ \int_{\R^N\backslash B_1} c^{s_R(x)}\,dx>0$ so that in view of \cite{libro}[Lemma 3.2.5],
$$ \|c\|_{L^{s_R(\cdot)}(\R^N\backslash B_1)}
\le \max \Big\{ \Big(\int_{\R^N\backslash B_1} c^{s_R(x)}\,dx \Big)^\frac{1}{s_R^-};
                      \Big(\int_{\R^N\backslash B_1} c^{s_R(x)}\,dx \Big)^\frac{1}{s_R^+} \Big\}. $$
As $ \int_{\R^N\backslash B_1} c^{s_R(x)}\,dx \to 0$ as $R\to \infty$ and
$\frac{1}{s_R^+}=\Big(\frac{1}{s_R}\Big)^- = \inf_{\R^N\backslash B_1} \frac{1}{p(\infty)} - \frac{1}{p_R} = 0$, we obtain
$ \|c\|_{L^{s_R(\cdot)}(\R^N\backslash B_1)}\le 1$.
In conclusion, for any $c\in (0,1)$ we have for $R$ big enough depending on $c$ only that
\begin{eqnarray*}
 \|\nabla u\|_{L^{p(\infty)}(\R^N\backslash B_1)} \le c^{-1}(1+o(1))  \|\nabla u\|_{L^{p_R(\cdot)}(\R^N\backslash B_1)}.
\end{eqnarray*}
\eqref{Holder1} follows.
\end{proof}

Observe that if, in the proof of prop \ref{PropIneq}, we take $\lambda\to +\infty$ with $|x_\lambda|\gg \lambda$, we can easily show that $S_\infty=0$ if $\infty$ is subcritical (in the sense that $q(\infty)<p(\infty)^*$). 
The details are given in the proof of the next proposition below. 
Conversely it seems natural to expect that $S_\infty>0$ if $\infty$ is critical. The proof happens to be non-trivial.
To see the difficulty, notice that the proof of this fact in the constant exponent case goes as follows: we assume that $K(n,p)^{-1}=0$ so that there exists a sequence $(u_k)_k$ of compactly supported smooth function such that 
$\int |\nabla u_k|^p\,dx \to 0$ and $\int |u_k|^{p*}\,dx =1$. Using the invariance by translation and dilatation of the norm, we can assume that each $u_k$ is supported in the unit ball $B$ and then use the embedding $W^{1,p}_0(B)\hookrightarrow L^{p^*}(B)$ to obtain a contradiction.

Since the invariance by translation and dilatation is lost in the variable exponent setting, this scheme of proof does not work anymore. Instead we will translate the problem to the sphere using the stereographic projection.

In the course of the proof we will need a Hardy inequality which is essentially a particular case of a general result proved in \cite{FMS}:

\begin{prop}
For any $v\in D^{1,p(\cdot)}(\R^N)$, $|x|^{-1}v\in
L^{p(\cdot)}(\R^N)$, and there exists a constant $C>0$ independent
of $v$ such that
\begin{equation}\label{Hardy}
 \||x|^{-1}v\|_{L^{p(\cdot)}(\R^N)}\le C\|\nabla v\|_{L^{p(\cdot)}(\R^N)}.
\end{equation}
\end{prop}

\begin{proof}
Let $I_1u(x)=\int_{\R^N}\frac{u(y)}{|x-y|^{N-1}}\,dy$ be the fractional integral of $u$.
As a particular case of  \cite{FMS}[Thm 4.1], we have the existence of a constant $C>0$ such that
for any $u\in L^{p(\cdot)}(\R^N)$, $|x|^{-1}I_1u\in L^{p(\cdot)}(\R^N)$ with
$$ \||x|^{-1}I_1u\|_{L^{p(\cdot)}(\R^N)}\le C\|u\|_{L^{p(\cdot)}(\R^N)}. $$
Taking $u=|\nabla v|$ with $v\in D^{1,p(\cdot)}(\R^N)$, we obtain
$$ \||x|^{-1}I_1(|\nabla v|)\|_{L^{p(\cdot)}(\R^N)}\le C\|\nabla v\|_{L^{p(\cdot)}(\R^N)}. $$
Independently, it is well-known (see e.g. \cite{GT}[Lemma 7.14]) that
for any $v\in C^\infty_c(\R^N)$, $|v|\le C(N)I_1(|\nabla v|)$ with a constant $C(N)$ depending only on $N$. The result easily follows.
\end{proof}

Our result is the following:

\begin{prop}\label{SinftyPositive}
 There holds that   $S_\infty>0$ if and only if $q(\infty)=p(\infty)^*$.
\end{prop}

\begin{proof}
Assume that $q(\infty)<p^*(\infty)$.
As in the proof of prop. \ref{PropIneq}, we fix some $\phi\in C^\infty_c(\R^N)$, $\phi\not\equiv 0$, and consider
for $\lambda\to +\infty$  the rescaled functions $\phi_\lambda$ defined by \eqref{rescaled}, where the points $x_\lambda\in \R^N$ are such that
$\lim_{\lambda\to +\infty}\frac{|x_\lambda|}{\lambda}=+\infty$.
It then follows from assumption (H1)-(H2) that for $y$ in a compact set, we can write \eqref{cont}-\eqref{cont2} and, since
$supp\,\phi_\lambda\subset\R^N-B_R$ for $\lambda$ big, also \eqref{Eq1}.
Then as before, we have
\begin{eqnarray*}
\int_{\R^N} |\phi_\lambda(x)|^{q(x)}\,dx
 = \lambda^{N\Big(1-\frac{ (q(\infty)+\eps_\lambda) }{ (p(\infty)+\eps_\lambda)^* }\Big)} \int_{\R^N} |\phi(y)|^{q(\infty)+\eps_\lambda}\,dy.
\end{eqnarray*}
As $\liminf_{\lambda\to +\infty}
1-\frac{q(\infty)+\varepsilon_\lambda}{(p(\infty)+\varepsilon_\lambda)^*}>0$, we have
$\int_{\R^N} |\phi_\lambda(x)|^{q(x)}\,dx\to\infty$.
On the other hand,  as before
$\lim_{\lambda\to +\infty} \int_{\R^N} |\nabla \phi_\lambda(x)|^{p(x)}\,dx
= \int_{\R^N} |\nabla\phi|^{p(\infty)}\,dy. $
We conclude that $S_\infty = 0$.

\medskip

Assume now that $q(\infty)=p(\infty)^*$.
Let $\Phi:S^N\backslash \{P\} \to \R^N$ be the stereographic projection centered at the north pole $P$ of the sphere $S^N$. It is well-known (see e.g. \cite{Lee}[lemma 3.4]) that the pull-back by $\Phi^{-1}$ of
the standard metric $h$ of $S^N$   is the metric conformal to the Euclidean metric $\xi$ given by
$$ (\Phi^{-1})^*h = \phi \xi \qquad \text{with} \qquad \phi(x)=4(1+|x|^2)^{-2}. $$
In particular for any function $u:\S^N\backslash \{P\}\to \R$ there hold
$$\int_{\S^N} u\,dv_h = \int_{\R^N} u\circ\Phi^{-1}\,dv_{(\Phi^{-1})^*h}
= \int_{\R^N} (u\circ\Phi^{-1}) \phi^{N/2} \,dx $$
and
$$ |\nabla u|_h^2 = |\nabla (u\circ \Phi^{-1})|_\xi^2 \circ \Phi. $$

To a function $u:\R^N\to \R$ with compact support, we associate a function
$\hat u: \S^N\to\R$  with compact support in $\S^N-\{P\}$ defined by
\begin{equation}\label{defHatu}
 \hat u = (u\phi^\beta)\circ\Phi \qquad \text{where} \qquad
2\beta := -\frac{N}{q(\infty)} = 1-\frac{N}{p(\infty)}.
\end{equation}
We also define the exponent $\hat p$ and $\hat q$ on $S^n$ by
\begin{equation}\label{defExp}
\hat p = p\circ\Phi \qquad \text{and}\qquad \hat q = q\circ\Phi.
\end{equation}

We first claim that if $u\in L^{q(\cdot)}(\R^N)$ is supported in $\R^N\backslash B_R$ then
\begin{equation}\label{NormLq}
 \int_{S^N} |\hat u|^{\hat q(\cdot)}\,dv_h = (1+o_R(1))\int_{\R^N} |u|^{q(\cdot)}\,dx,
\end{equation}
where $\lim_{R\to +\infty} o_R(1)=0$.
Indeed
$$ \int_{S^N} |\hat u|^{\hat q(\cdot)}\,dv_h
  = \int_{\R^N} |u\phi^\beta|^{q(\cdot)} \, dv_{\phi \xi}
  = \int_{\R^N\backslash B_R} |u|^{q(\cdot)} \phi^{\beta q(\cdot)+N/2}\,dx $$
with $\beta q(x)+N/2=\frac{n}{2q(\infty)}(q(\infty)-q(x))$.
Using (H2), it is easy to see that $\lim_{|x|\to +\infty}\phi(x)^{\beta q(x)+N/2}=1$ which proves the claim.

We now claim that there exists a constant $C>0$ such that for any $u\in C^1(\R^N)$,
\begin{equation}\label{NormGrad}
\int_{\S^N} |\nabla \hat u|_h^{\hat p(\cdot)}\, dv_h
\le C \int_{\R^N} |\nabla u|^{p(\cdot)}\, dx
 + C \int_{\R^N} \frac{|u|^{p(\cdot)}}{|x|^{p(\cdot)}}\,dx.
\end{equation}
To prove this we begin writing that
\begin{eqnarray*}
&& \int_{\S^N} |\nabla \hat u|_h^{\hat p(\cdot)}\, dv_h
 =  \int_{\R^N} |\nabla (u\phi^\beta)|^{p(\cdot)}_{\phi\xi}\, dv_{\phi\xi}
= \int_{\R^N} |\nabla (u\phi^\beta)|^{p(\cdot)}_\xi \phi^\frac{N-p(\cdot)}{2}  \, dx \\
& \le &
C \int_{\R^N} |\nabla u|^{p(\cdot)} \phi^{\beta p(\cdot)+\frac{N-p(\cdot)}{2}}  \, dx
+ C \int_{\R^N} |u|^{p(\cdot)}|\nabla \phi|^{p(\cdot)} \phi^{(\beta-1)p(\cdot)+ \frac{N-p(\cdot)}{2}}  \, dx \\
& \le &
C \int_{\R^N} |\nabla u|^{p(\cdot)} \phi^{\frac{N}{2p(\infty)}(p(\infty)-p(x))}  \, dx
+ C \int_{\R^N} |u|^{p(\cdot)}\Big(\frac{|\nabla \phi|}{\phi}\Big)^{p(\cdot)} \phi^{\frac{N}{2p(\infty)}(p(\infty)-p(x))}  \, dx.
\end{eqnarray*}
As before, using (H2) we see that
$\lim_{|x|\to +\infty}\phi(x)^{\frac{N}{2p(\infty)}(p(\infty)-p(x))} = 1$.
We thus obtain
\begin{eqnarray*}
\int_{\S^N} |\nabla \hat u|_h^{\hat p(\cdot)}\, dv_h
& \le &
 \int_{\R^N} |\nabla u|^{p(\cdot)} \, dx + \int_{\R^n} |u|^{p(\cdot)}\Big(\frac{|\nabla \phi|}{\phi}\Big)^{p(\cdot)} \, dx.
\end{eqnarray*}
Moreover direct computation shows that $\frac{|\nabla \phi|}{\phi} =
\frac{4|x|}{1+|x|^2}\le \frac{C}{|x|}$. The result follows.

\medskip

We can now prove that $S_\infty>0$. Assume on the contrary that $S_\infty=0$. In particular $S_R=0$ for any $R>0$.
We can thus find smooth functions $u_R$, $R>0$, compactly supported in $\R^N\backslash B_R$ such that
\begin{equation}\label{Contradiction}
 \lim_{R\to +\infty}\|\nabla u_R\|_{p(\cdot)} = 0 \qquad \text{and} \qquad
 \|u_R\|_{q(\cdot)} =1.
\end{equation}
Then in view of \eqref{NormLq}  and \eqref{NormGrad},  the functions $\hat u_R$ defined on $\S^N$ from $u_R$ according to \eqref{defHatu} satisfies
\begin{equation}\label{Lqnormzero}
 \int_{S^N} |\hat u_R|^{\hat q}\,dv_h = 1+o_R(1)
\end{equation}
 and
\begin{equation*}
\int_{\S^N} |\nabla \hat u_R|_h^{\hat p(\cdot)}\, dv_h
\le o_R(1) + C \int_{\R^N} \frac{|u_R|^{p(\cdot)}}{|x|^{p(\cdot)}}\,dx
\end{equation*}
Using Hardy inequality \eqref{Hardy}, we see that the integral in
the right hand side is $o_R(1)$ since
\begin{eqnarray*}
 \int_{\R^N} \frac{|u_R|^{p(\cdot)}}{|x|^{p(\cdot)}}\,dx
& \le & \max\,\{\||x|^{-1}u_R\|_{L^{p(\cdot)}(\R^N)}^{p^+},\||x|^{-1}u_R\|_{L^{p(\cdot)}(\R^N)}^{p^-} \}  \\
& \le & C \max\,\{\|\nabla u_R\|_{L^{p(\cdot)}(\R^N)}^{p^+},\|\nabla u_R\|_{L^{p(\cdot)}(\R^N)}^{p^-} \} \\
& = & o_R(1).
\end{eqnarray*}
We thus deduce that
\begin{equation}\label{NormGrad2}
\int_{\S^N} |\nabla \hat u_R|_h^{\hat p(\cdot)}\, dv_h = o_R(1).
\end{equation}
We now obtain a contradiction between \eqref{Lqnormzero} and
\eqref{NormGrad2}. To see this denote by $\Psi:\S^N\backslash
\{S\}\to \R^N$ the stereographic proyection from the south pole $S$,
and consider the functions $\hat u_R\circ \Psi^{-1}$ and the
exponents $\hat p\circ \Psi^{-1}$ and $\hat q\circ \Psi^{-1}$.
Recall that the $\hat u_R$ are supported in ball of $\S^N$ centered
at $P$ whose radius becomes smaller and smaller as $R\to +\infty$.
Since the pullback by $\Psi^{-1}$ of $h$ restricted to these balls
is a metric in $B_1\subset \R^N$ bounded above and below by
$(1+\eps)\xi$ and $(1-\eps)\xi$ for arbitrary small $\eps$, we
obtain as before that
$$ \int_{B_1}|\nabla (\hat u_R\circ \Psi^{-1})|^{\hat p\circ \Psi^{-1}}\,dx = o_R(1), $$
and
$$ \int_{B_1}|\hat u_R\circ \Psi^{-1}|^{\hat q\circ \Psi^{-1}}\,dx  \ge 1+o_R(1).$$
Recalling that $\hat u_R\circ \Psi^{-1}\in C^\infty_c(B_1)$ for $R$ big, this contradicts the Poincare inequality in $B_1$.
\end{proof}

\section{Concentration-compactness principle at infinity}

This section is devoted to the description of the default of compactness of a sequence of functions $(u_n)_n\subset D^{1,p(\cdot)}(\R^N)$,
the closure of $C^\infty_c(\R^N)$ for the norm $\|\nabla u\|_{L^{p(\cdot)}(\R^N)}$.
This will be done by establishing a version of Lions'concentration-compactness principle (CCP) in that setting.

The CCP originally due to P.L. Lions \cite{Lions} was established in the framework of constant-exponent spaces over a bounded domain.
It explains the possible loss of compactness of the $u_n$ by their weak convergence to Dirac masses.
The weights of the Dirac masses are related to the best sobolev constant \eqref{BestSobCste}.
It was later extended by J. Chabrowski \cite{Chabrowski} to unbounded domain, still for constant-exponent spaces,
by introducing measures supported at infinity  related to  the best sobolev constant \eqref{BestSobCste}.
An extension to variable exponent spaces over a bounded domain was done independently in \cite{FBS1} and \cite{Fu}, and
over unbounded domain by \cite{Fu2}.
In those papers the weights of the Dirac masses  are related to the Sobolev constant
$\inf_{u\in C^\infty_c(\R^N)} \frac{\|\nabla u\|_{L^{p(\cdot)}(\R^N)}}{\|u\|_{L^{q(\cdot)}(\R^N)}}$.
Notice that this Sobolev constant is global in the sense that it depends on the behaviour of $p$ and $q$ in all $\R^N$.
This is not satisfactory since a Dirac mass is supported at a single point and thus
the weights should depend only on the behaviour of $p$ and $q$ near that point.
A refined version of the CCP was proven in \cite{FBSS1} where the weights of the Dirac masses are related to the localized Sobolev constant \eqref{DefLocCste}.

The purpose of this section is to state a CCP in $\R^N$ where the weight of the "`Dirac mass at infinity"' is related to the
localized Sobolev constant at infinity \eqref{DefSinf}.

\begin{teo}\label{propCCP}
 Let $(u_n)\subset D^{1,p(\cdot)}(\R^N)$ be a weakly convergent  to $u\in D^{1,p(\cdot)}(\R^N)$.
Then there exist two bounded measures $\mu$ and $\nu$, an at most enumerable set of indices $I$, 
points $x_i\in\A$ (the critical set defined in (H3)),
and positive real numbers $\mu_i,\nu_i$, $i\in I$, such that the following convergence hold weakly in the sense of measures,
\begin{eqnarray}\label{CCP}
 |\nabla u_n|^{p(x)}\,dx & \rightharpoonup & \mu\ge |\nabla u|^{p(x)}\,dx + \sum \mu_i \delta_{x_i}, \label{CCP1} \\
|u_n|^{q(x)}\,dx & \rightharpoonup & \nu:= |u|^{q(x)}\,dx + \sum \nu_i \delta_{x_i}, \label{CCP2}  \\
  S_{x_i}\nu_i^\frac{1}{p(x_i)^*} & \le & \mu_i^\frac{1}{p(x_i)}\qquad \text{for all }i\in I, \label{CCP3}
\end{eqnarray}
where $S_{x_i}$ is the localized Sobolev constant at the point $x_i$ defined in \eqref{DefLocCste}.
Moreover, if we define
$$ \nu_\infty=\lim_{R\to\infty}\limsup_{n\to\infty}\int_{|x|>R}|u_n|^{q(x)}\,dx,  $$
$$ \mu_\infty=\lim_{R\to\infty}\limsup_{n\to\infty}\int_{|x|>R}|\nabla u_n|^{p(x)}\,dx, $$
then
\begin{eqnarray}
 \limsup_{n\to\infty} \int_{\R^N} |\nabla u_n|^{p(x)}\,dx & = & \mu(\R^N) + \mu_\infty, \label{CCPinf1} \\
\limsup_{n\to\infty} \int_{\R^N} |u_n|^{q(x)}\,dx & = & \nu(\R^N) + \nu_\infty, \label{CCPinf2} \\
 S_\infty\nu_\infty^\frac{1}{q(\infty)} & \leq & \mu_\infty^\frac{1}{p(\infty)},  \label{CCPinf3}
\end{eqnarray}
where $S_\infty$ is the localized Sobolev constant at infinity defined in (\ref{DefSinf}).
\end{teo}

The first part of the theorem, namely \eqref{CCP1}-\eqref{CCP3}, was proved in \cite{FBSS1}.
The 2nd part, namely \eqref{CCPinf1}-\eqref{CCPinf3} was proved in \cite{Fu2} but with a global Sobolev constant in \eqref{CCP3}.
The main contribution of the present theorem is inequality \eqref{CCPinf3} with the presence of the localized Sobolev constant at infinity.
This will allow us to formulate in the next sections a local condition at infinity for the existence of a solution to \eqref{MainEqu}.

\begin{proof}

We refer to  \cite{FBSS1} for the proof of \eqref{CCP1}-\eqref{CCP3} and concentrate on \eqref{CCPinf1}-\eqref{CCPinf3}.

\medskip

A detailed proof of \eqref{CCPinf1}-\eqref{CCPinf2} can be found in \cite{Fu2}[thm 2.5]. We briefly sketch it for the sake of completeness.

Consider a smooth function $\phi:[0,+\infty)\to [0,1]$ such that $\phi\equiv 0$ in $[0,1]$ and $\phi\equiv 1$ in $[2,+\infty)$.
Then $\phi_R(x):=\phi(|x|/R)$ is smooth and satisfies $\phi_R(x)=1$ for $|x|\ge 2R$, $\phi_R(x)=0$ for $|x|\le R$ and $0\leq\phi_R(x)\leq1$.
We then write that
\begin{equation}\label{mu1}
  \int_{\R^N}|\nabla u_n|^{p(x)}\,dx=\int_{\R^N}|\nabla u_n|^{p(x)}\phi_R\,dx+\int_{\R^N}|\nabla u_n|^{p(x)}(1-\phi_R)\,dx.
\end{equation}
Observe first that
$$ \int_{\{|x|>2R\}}|\nabla u_n|^{p(x)}\,dx\leq\int_{\R^N}|\nabla u_n|^{p(x)}\phi_R^{p(x)}\,dx
\leq\int_{\{|x|>R\}}|\nabla u_n|^{p(x)}\,dx
$$
so that
\begin{equation}\label{mu2}
\mu_\infty=\lim_{R\to\infty}\limsup_{n\to\infty}\int_{\R^N}|\nabla u_n|^{p(x)}\phi_R^{p(x)}\,dx.
\end{equation}
In the same way
\begin{equation}\label{nu2}
\nu_\infty=\lim_{R\to\infty}\limsup_{n\to\infty}\int_{\R^N}|u_n|^{q(x)}\phi_R^{q(x)}\,dx.
\end{equation}
On the other hand, since $1-\phi_R$ is smooth with compact support, we have by definition of $\mu$ that
for $R$ fixed,
$$ \lim_{n\to +\infty}  \int_{\R^N}(1-\phi_R)|\nabla u_n|^{p(x)}\,dx = \int_{\R^N}(1-\phi_R)\,d\mu.  $$
Noticing that $ \lim_{R\to +\infty}\int_{\R^N} \phi_R\,d\mu = 0 $ by dominated convergence, we obtain
\begin{equation}\label{mu3}
\lim_{R\to\infty}\limsup_{n\to\infty}\int_{\R^N} (1-\phi_R)|\nabla u_n|^{p(x)}\,dx = \mu(\R^N).
\end{equation}
Plugging \eqref{mu2} and \eqref{mu3} into \eqref{mu1} yields \eqref{CCPinf1}. The proof of \eqref{CCPinf2} is similar.

\medskip

We prove (\ref{CCPinf3}).
By definition of  $S_R=S(p(\cdot),q(\cdot),\R^N\setminus B_R)$ we can write that
\begin{eqnarray}\label{Eq4}
S_R\|u_n\phi_R\|_{L^{q(\cdot)}(\R^N\setminus B_R)}
& \leq & \|\nabla (u_n \phi_R)\|_{L^{p(\cdot)}(\R^N\setminus B_R)} \nonumber\\
& \leq & \|(\nabla u_n) \phi_R\|_{L^{p(\cdot)}(\R^N\setminus B_R)}
+\|u_n \nabla\phi_R\|_{L^{p(\cdot)}(B_{R+1}\setminus B_R)}.
\end{eqnarray}
Observe that since $u_n\to u$ in $L^{p(\cdot)}_{loc}(\R^N)$ we have
\begin{eqnarray*}
\limsup_{n\to +\infty}\int_{B_{2R}\setminus B_R} |u_n|^{p(x)}|\nabla\phi_R|^{p(x)}\,dx
&= &\int_{B_{2R}\setminus B_R} |u|^{p(x)}|\nabla\phi_R|^{p(x)}\,dx \\
&\le & C\||u|^{p(\cdot)}\|_{L^\frac{p^*(\cdot)}{p(\cdot)}(B_{2R}\setminus B_R)}
      \|  |\nabla\phi_R|^{p(\cdot)} \|_{L^\frac{N}{p(\cdot)}(\R^N)} \\
&\le & C\||u|^{p(\cdot)}\|_{L^\frac{p*(\cdot)}{p(\cdot)}(B_{2R}\setminus B_R)}.
\end{eqnarray*}
Since $u\in L^{p^*(\cdot)}$ we deduce that
\begin{equation}\label{EQ100}
\lim_{R\to +\infty}\limsup_{n\to +\infty} \|u_n \nabla\phi_R\|_{L^{p(\cdot)}(B_{2R}\setminus B_R)}=0.
\end{equation}
Independently, letting $p_R^+=\sup_{|x|\ge R}p(x)$ and
$p_R^-=\inf_{|x|\ge R}p(x)$, we have
\begin{align*}
\|(\nabla u_n) \phi_R\|_{L^{p(\cdot)}(\R^N\setminus B_R)}
&\leq\max\left\{\left(\int_{\R^N}|(\nabla u_n)\phi_R|^{p(x)}\,dx\right)^\frac{1}{p_R^+},
\left(\int_{\R^N}|(\nabla u_n)\phi_R|^{p(x)}\,dx\right)^\frac{1}{p_R^-}\right\}
\end{align*}
Given $\eps>0$ we then have for $R$ big that
\begin{align*}
\|(\nabla u_n) \phi_R\|_{L^{p(\cdot)}(\R^N\setminus B_R)}
&\leq\max\left\{\left(\int_{\R^N}|(\nabla u_n)\phi_R|^{p(x)}\,dx\right)^\frac{1}{p(\infty)-\varepsilon},\left(\int_{\R^N}|(\nabla u_n)\phi_R|^{p(x)}\,dx\right)^\frac{1}{p(\infty)+\varepsilon}\right\}.
\end{align*}
Analogously,
\begin{align*}
\|u_n\phi_R\|_{L^{q(\cdot)}(\R^N\setminus B_R)}
&\ge \min\left\{\left(\int_{\R^N}|u_n\phi_R|^{q(x)}\,dx\right)^\frac{1}{q^+},\left(\int_{\R^N}|u_n\phi_R|^{q(x)}\,dx\right)^\frac{1}{q^-}\right\}\\
&\geq \min\left\{\left(\int_{\R^N}|u_n\phi_R|^{q(x)}\,dx\right)^\frac{1}{q(\infty)-\varepsilon},\left(\int_{\R^N}|u_n\phi_R|^{q(x)}\,dx\right)^\frac{1}{q(\infty)+\varepsilon}\right\}.
\end{align*}
Assume for instance that there exists $R_0$ such that
$\int_{\R^N}|(\nabla u_n)\phi_{R_0}|^{p(x)}\,dx<1$
and $\int_{\R^N}|u_n\phi_{R_0}|^{q(x)}\,dx<1$
so that $\int_{\R^N}|(\nabla u_n)\phi_R|^{p(x)}\,dx<1$ and
$\int_{\R^N}|u_n\phi_{R_0}|^{q(x)}\,dx<1$  for all $R>R_0$.
The others cases can be handled similarly. Then \eqref{Eq4} gives for $R$ big,
$$
S_R\left(\int_{\R^N}|u_n\phi_R|^{q(x)}\,dx\right)^\frac{1}{q(\infty)-\varepsilon}\leq\left(\int_{\R^N}|(\nabla u_n)\phi_R|^{p(x)}\,dx\right)^\frac{1}{p(\infty)+\varepsilon}
+ \|u_n \nabla\phi_R\|_{L^{p(\cdot)}(B_{2R}\setminus B_R)}.
$$
Taking limit as $n\to\infty$ and then as $R\to\infty$, we obtain in view of \eqref{mu2}, \eqref{nu2} and \eqref{EQ100} that
$$ S_\infty\nu_\infty^\frac{1}{q(\infty)-\varepsilon} \leq\mu_\infty^\frac{1}{p(\infty)+\varepsilon} $$
for any $\eps>0$. We thus deduce (\ref{CCPinf3}).

\end{proof}

\section{Local existence condition for equation \eqref{MainEqu}.}

In this section we use the CCP at infinity Theorem \ref{propCCP} to obtain local existence condition for  equation \eqref{MainEqu} namely
$$ -\Delta_{p(x)}u +k(x)|u|^{p(x)-2}u = K(x)|u|^{q(x)-2}u  \qquad \text{ in }\R^N. $$
We will look for a solution as a critical point of the associated functional
\begin{equation}\label{DefF}
\F(u):=\int_{\R^N}\frac{|\nabla u|^{p(x)}}{p(x)}-K(x)\frac{|u|^{q(x)}}{q(x)}+  k(x)\frac{|u|^{p(x)}}{p(x)}\,dx.
\end{equation}
Besides assumptions (H1), (H2), (H3) we will also assume in this section that
\begin{itemize}
\item[(H4)] the nonlinearity is superlinear in the sense that $q^->p^+$ where $p^+=\sup_{\R^N}p$ and $q^-=\inf_{\R^N}q$.
\item[(H5)] $k:\R^N\to \R$ is continuous bounded and the functional $u\to \int_{\R^N} |\nabla u|^{p(x)}+ k(x)|u|^{p(x)}\,dx$
            is coercive in the sense that the expression
\begin{equation}\label{DefNorm}
 \|u\|:=\inf\,\Big\{\lambda\ge 0\, \text{ s.t. } \int_{\R^N}  \frac{|\nabla u|^{p(x)} + k(x) |u|^{p(x)} }{\lambda} \,dx\le 1 \Big\}
\end{equation}
defines a norm equivalent to the standard norm of $W^{1,p(\cdot)}(\R^N)$.
\item[(H6)] the embedding from $D^{1,p(\cdot)}(\R^N)$ into $L^{p(\cdot)}(\R^N,k\,dx)$ is compact.
\item[(H7)] $K:\R^N\to \R$ is continuous non-negative and has a limit at infinity $K(\infty):=\lim_{|x|\to +\infty}K(x)$.
\end{itemize}
A  sufficient condition for assumption (H6) to hold is given in  prop. \ref{Chab} in Appendix A.

Recall that a sequence $\{u_n\}_{j\in\N}\subset W^{1,p(\cdot)}(\R^N)$ is a Palais-Smale sequence (P-S) for $\F$ if the sequence
$(\F(u_n))_n$ is bounded and $\F'(u_n)\to 0$ in $W^{1,p(\cdot)}(\R^N)'$.
Moreover $\F$ is said to satisfy the P-S condition at level $c$ if any P-S sequence $\{u_n\}_{j\in\N}\subset W^{1,p(\cdot)}(\R^N)$
for $\F$ such that $\F(u_n)\to c$ has a subsequence strongly convergent in $W^{1,p(\cdot)}(\R^N)$.
We will prove that $\F$ satisfies the P-S condition at level small enough depending on the localized Sobolev constant 
$S_x$, $x\in\A$, and $S_\infty$. 

We first prove a  preliminary lemma which is more or less classical. 

\begin{lema}\label{acotada1}
Let $\{u_n\}_{n\in\N}\subset W^{1,p(\cdot)}(\R^N)$ be a Palais-Smale sequence for $\F$.
Then, up to a subsequence, there exists $u\in  W^{1,p(\cdot)}(\R^N)$ such that $u_n\to u$ weakly in $W^{1,p(\cdot)}(\R^N)$ 
and $u$ is a weak solution of \eqref{MainEqu} with $\F(u)\ge 0$. 

Moreover letting $\mu$, $\nu$, $\mu_i$, $\nu_i$, $\mu_\infty$, $\nu_\infty$
be as in the concentration-compactness principle  Theorem \ref{propCCP} when applied to $(u_n)_n$ we have the following estimates: 
\begin{eqnarray}
 \nu_i &\ge &  S_{x_i}^N K(x_i)^{-\frac{N}{p(x_i)}}, \quad  \mu_i \ge   S_{x_i}^N K(x_i)^{1-\frac{N}{p(x_i)}}
            \qquad \text{if } K(x_i)>0,  \label{CotaInf1} \\
\mu_i & =& \nu_i=0 \qquad \text{ if } K(x_i)= 0, \label{CotaInf2}
\end{eqnarray}
and a similar result at infinity:
\begin{eqnarray}
 \nu_\infty &\ge &  S_\infty^N K(\infty)^{-\frac{N}{p(\infty)}}, \quad 
\mu_\infty \ge  S_\infty^N K(\infty)^{1-\frac{N}{p(\infty)}}
            \qquad \text{if } K(\infty)>0 \label{CotaInf3}  \\
\mu_\infty & =& \nu_\infty=0 \qquad \text{ if } K(\infty)= 0. \label{CotaInf4}
\end{eqnarray}
\end{lema}

\begin{proof}
The proof is more or less classical so we will be sketchy. 
First, recalling the definition of a PS sequence, it is easily seen that 
\begin{align*}
C+o(1)\|u_n\|_{W^{1,p(\cdot)}(\R^N)}
&\geq \mathcal{F}(u_n)-\frac{1}{q^-}\langle\mathcal{F}'(u_n),u_n\rangle\\
&\geq \left(\frac{1}{p^+} - \frac{1}{q^-}\right) \int_{\R^N} |\nabla u_n|^{p(x)}+ k|u_n|^{p(x)}\, dx\\
\end{align*}
In view of assumptions (H4) and (H5), we deduce that $\{u_n\}_{n\in\N}$ is bounded in $W^{1,p(\cdot)}(\R^N)$.
Up to a subsequence we can thus assume that $(u_n)_n$ weakly converge in $W^{1,p(\cdot)}(\R^N)$ to some $u$.  

Estimates \eqref{CotaInf1}-\eqref{CotaInf4} are a direct consequence of \eqref{CCP3}, \eqref{CCPinf3} and the following:
\begin{equation}\label{Cota}
\mu_i = \nu_i K(x_i) \qquad \text{for any }i\in I,
\end{equation}
 and
\begin{equation}\label{CotaInfinity}
\mu_\infty = K(\infty)\nu_\infty.
\end{equation}

To prove \eqref{Cota}, we fix a concentration point $x_i$, a smooth function
$\phi:\R^N\to [0,1]$  with compact support in $B_2$  such that $\phi=1$ in $B_1$,
and consider $\phi_\delta(x):=\phi(|x-x_i|/\delta)$.
We then write that
\begin{eqnarray*}
o(1) &=& \langle  \F'(u_n), u_n\phi_\delta\rangle \nonumber \\
& = & \int_{\R^N} |\nabla u_n|^{p(x)}\phi_\delta\, dx - \int_{\R^N} K|u_n|^{q(x)}\phi_\delta\, dx
+ \int_{\R^N}  u_n|\nabla u_n|^{p(x)-2}\nabla u_n \nabla\phi_\delta\, dx \\
&& + \int_{\R^N} k|u_n|^{p(x)}\phi_\delta\, dx. 
\end{eqnarray*}
Since $u_n\to 0$ in $L_{loc}^{p(\cdot)}(\R^N)$ we obtain by passing to the limit $n\to +\infty$ that 
\begin{equation}\label{Equ300}
  \int_{\R^N}\phi_\delta\, d\mu - \int_{\R^N} K \phi_\delta\,d\nu + \int_{\R^N} k|u|^{p(x)}\phi_\delta\, dx
	 =  - \lim_{n\to +\infty} \int_{\R^N}  u_n|\nabla u_n|^{p(x)-2}\nabla u_n \nabla\phi_\delta\, dx. 
\end{equation}
Observe that 
\begin{eqnarray*}
 \Big| \int_{\R^N}  (u_n-u)|\nabla u_n|^{p(x)-2}\nabla u_n \nabla\phi_\delta\, dx \Big| 
&\le & \frac{C}{\delta } \|u_n-u\|_{L^{p(\cdot)}(B_{x_i}(2\delta))} \||\nabla u_n|^{p(x)-1}\|_{L^\frac{p(\cdot)}{p(\cdot)-1}(\R^N)}  \\ 
& \le& \frac{C}{\delta} \|u_n-u\|_{L^{p(\cdot)}(B_{x_i}(2\delta))} 
\end{eqnarray*} 
which goes to $0$ as $n\to +\infty$ for $\delta$ fixed.  Moreover 
\begin{eqnarray*}
 \Big| \int_{\R^N}  u|\nabla u_n|^{p(x)-2}\nabla u_n \nabla\phi_\delta\, dx \Big| 
& \le &  \frac{C}{\delta }  \||\nabla u_n|^{p(x)-1}\|_{L^\frac{p(\cdot)}{p(\cdot)-1}(\R^N)} \|u\|_{L^{p^*(\cdot)}(B_{x_i}(2\delta))} 
     \|1\|_{L^N(B_{x_i}(2\delta))}  \\ 
& \le & C \|u\|_{L^{p^*(\cdot)}(B_{x_i}(2\delta))}  
\end{eqnarray*} 
where the constant $C$ does not depend on $n$. 
We thus obtain from \eqref{Equ300} that 
\begin{equation*}
  \int_{\R^N}\phi_\delta\, d\mu - \int_{\R^N} K \phi_\delta\,d\nu + \int_{B_{x_i}(2\delta)} k|u|^{p(x)}\phi_\delta\, dx
	 =  O\Big( \|u\|_{L^{p^*(\cdot)}(B_{x_i}(2\delta))}  \Big).  
\end{equation*} 
Letting $\delta\to 0$ then gives \eqref{Cota}.

To prove \eqref{CotaInfinity}, we fix $\phi:\R^N\to [0,1]$ smooth such that $\phi\equiv 0$ in $B_1$
and $\phi\equiv 1$ in $\R^N\backslash B_2$, and then consider $\phi_R(x):=\phi(x/R)$.
Then
\begin{eqnarray}\label{Cota10}
o(1) & = & \langle \mathcal{F}'(u_n), u_n \phi_R\rangle \nonumber \\
& =&  \int_{\R^N}|\nabla u_n|^{p(x)}\phi_R^{p(x)}\,dx
- \int_{\R^N}K|u_n|^{q(x)}\phi_R^{q(x)} \,dx
+ \int_{\R^N} u_n|\nabla u_n|^{p(x)-2}\nabla u_n\nabla\phi_R \,dx\nonumber \\
&&+\int_{\R^N} k |u_n|^{p(x)}\phi_R^{p(x)}\, dx. \nonumber \\
& =: & A_n+B_n-C_n+D_n.
\end{eqnarray}
In view of \eqref{mu2} and \eqref{nu2},
$$ \lim_{R\to +\infty}\limsup_{n\to +\infty}\, A_n= \mu_\infty \qquad \text{and} \qquad
\lim_{R\to +\infty}\limsup_{n\to +\infty}\, B_n= K(\infty)\nu_\infty. $$
Independently, since $(u_n)$ is bounded in  $W^{1,p(\cdot)}(\R^N)$ and $\|\nabla \phi_R\|_\infty\le C/R$,
the third integral is bounded by $ C/R$ so that
$$ \lim_{R\to +\infty}\limsup_{n\to +\infty}\, C_n= 0.$$
Eventually in view of assumption (H6),
$$ \lim_{R\to +\infty}\limsup_{n\to +\infty}\, C_n
= \lim_{R\to +\infty}\int_{\R^N} k |u|^{p(x)}\phi_R^{p(x)}\, dx
= 0. $$
Thus letting $n\to +\infty$ and then $R\to +\infty$ in \eqref{Cota10} yields \eqref{CotaInfinity}.

\medskip 

It follows from \eqref{CotaInf1}-\eqref{CotaInf2} that there are at most a finite number of concentration points.   
Mimicking the proof of \cite{FuZhang2}[Thm 3.1], we can prove that $\nabla u_n\to \nabla u$ a.e. 
and then that $u$ is a weak solution of \eqref{MainEqu}.  
Recalling that  $p^+<q^-$ and that $K$ is non-negative it is then easily seen that $\F(u)\ge 0$. Indeed 
\begin{equation*}
\begin{split}
 \F(u) & \ge \frac{1}{p^+} \int_{\R^N} \left(|\nabla u|^{p(x)}+ k(x)|u|^{p(x)}\right)\,dx
           - \frac{1}{q^-} \int_{\R^N} K(x) |u|^{q(x)}\,dx  \\
      & = \left( \frac{1}{p^+} - \frac{1}{q^-} \right) \int_{\R^N} K(x) |u|^{q(x)}\,dx 
\end{split}
\end{equation*}
which is non-negative. 
\end{proof}

The previous lemma gives easily the following

\begin{prop}\label{PropPSCond}
The functional $\F$ defined in \eqref{DefF} verifies the PS condition at level $c$ for all real numbers $c$ satisfying
$$ c<\inf_{x\in \A\cup\{\infty\},\, K(x)>0} \frac1N S_{x}^N K(x)^{1-\frac{N}{p(x)}}.$$
\end{prop}

\begin{proof}

Let $(u_n)_n$ be a PS sequence for $\F$ of level $c$. Up to a
subsequence, we can assume that $(u_n)_n$ weakly converges to some
$u$. Let $\mu$, $\nu$, $\mu_i$, $\nu_i$, $\mu_\infty$, $\nu_\infty$
be as in the concentration-compactness principle  Theorem.
\ref{propCCP} when applied to $(u_n)_n$.
In view of the definition of $\mu$ and $\nu$ and assumption (H6),
\begin{eqnarray*}
 c & = &  \lim_{n\to\infty} \F(u_n) \\
 & = & \int_{\R^N} \frac{1}{p(x)}\,d\mu + \int_{\R^N}\frac{k(x)}{p(x)}|u|^{p(x)}\,dx - \int_{\R^N} \frac{1}{q(x)}\,d\nu \\
   &\ge &  \F(u) + \sum_i \Big\{\frac{\mu_i}{p(x_i)} - \frac{K(x_i)\nu_i}{p(x_i)^*}\Big\}
         +  \frac{\mu_\infty}{p(\infty)} - \frac{K(\infty)\nu_\infty}{q(\infty)}.
\end{eqnarray*}
Since $\F(u)\ge 0$, we deduce,  using \eqref{Cota}, \eqref{CotaInfinity} and \eqref{CotaInf1}-\eqref{CotaInf4}, that
\begin{equation*}
c  \ge \sum_{x\in \{x_i,\,i\in I\}\cup \{\infty\},\,K(x)>0}   \frac{1}{N} S_{x}^N K(x)^{1-\frac{N}{p(x)}}
 \end{equation*}
from which the result follows.
\end{proof}

A direct application of the Mountain-pass theorem combined with
proposition \ref{PropPSCond} then yields the following existence
condition:

\begin{teo}\label{teoMP} Assume (H1)-(H7). 
If there exists $v\in W^{1,p(\cdot)}(\R^N)$ such that
\begin{equation}\label{CCPCond}
 \sup_{t>0} \F(tv) < \inf_{x\in \A\cup\{\infty\},\, K(x)>0}
\frac1N S_x^N K(x)^{1-\frac{N}{p(x)}},
\end{equation}
then \eqref{MainEqu} has a non-trivial  solution.
\end{teo}

\noindent Notice that if $\infty$ is subcritical, i.e.
$q(\infty)<p(\infty)^*$, then $S_\infty=0$ in view of prop.
\ref{SinftyPositive}. In that case the right hand side in
\eqref{CCPCond} becomes non-positive and the Theorem becomes
useless. The reason is that \eqref{CotaInf3} reduces to
$\mu_\infty,\nu_\infty\ge 0$ which gives no information about the
loss of mass at infinity. To solve this problem, the definition of
$S_\infty$ must be modified. The results we obtained in that case
will be presented elsewhere.

\begin{proof}
Notice first that $\F$ has the mountain-pass geometry. Indeed $\F(0)=0$ and, if $\|u\|_{W^{1,p(\cdot)}(\R^N)}=r$ is small enough, then
using the coercivity assumption (H5) and the Sobolev inequality, we have
$$ \F(u) \ge C\|u\|_{W^{1,p(\cdot)}(\R^N)}^{p^+} - C\max\{\|u\|_{q(\cdot)}^{q^+},\|u\|_{q(\cdot)}^{q^-}\}
         \ge Cr^{p^+} - Cr^{q^-} $$
which is positive by (H4).

Let $v\in W_0^{1,p(x)}(\R^N)$ satisfying \eqref{CCPCond}.
It easily follows from (H4) that $\F(tu)<0$ for $t>0$ big.
In view of proposition \ref{PropPSCond}, we can then apply the
Mountain-pass theorem to obtain a critical point of $\F$.
\end{proof}

\begin{prop}
The infimum in the r.h.s. of \eqref{CCPCond} is attained at some point of $\A\cup\{\infty\}$.
\end{prop}

\begin{proof}
We first prove that the function $x\in\A\to S_x$ is lower semi-continuous. 
Consider $x_0\in\A$, $(x_n)_n\subset \A$ such that $x_n\to x_0$ and fix some $\eps>0$. 
There exists $n(\eps)\in \N$ such that  $B_{x_n}(\eps/3)\subset B_{x_0}(\eps)$ for $n\ge n(\eps)$. 
It follows that 
$$ S(p(\cdot),q(\cdot),B_{x_0}(\eps))\le S(p(\cdot),q(\cdot),B_{x_n}(\eps/3)) \le S_{x_n} $$ 
for $n\ge n(\eps)$. Then $\liminf_{n\to +\infty} S_{x_n} \ge  S(p(\cdot),q(\cdot),B_{x_0}(\eps))$ for any $\eps$. 
Letting $\eps\to 0$ gives $\liminf_{n\to +\infty} S_{x_n} \ge S_{x_0}$. 

To prove the proposition, it thus suffices to prove that this function is
also lower semi-continuous at infinity in the sense that for any
sequence $x_n\in\A$ such that $|x_n|\to +\infty$, there holds
$\liminf_{n\to +\infty} S_{x_n}\ge S_\infty$. Fix some $R>0$ and
then $n_0\in \mathbb{N}$ such that $|x_n|\ge R+1$. Then for $n\ge
n_0$, $B_{x_n}(\eps)\subset \R^N\backslash B_R$ for any $\eps<1$. It
follows that for such $n$ and $\eps$,
$S(p(\cdot),q(\cdot),B_{x_n}(\eps))\ge
S(p(\cdot),q(\cdot),\R^N\backslash B_R)$. Taking the limit in $\eps$
and then in $n$ gives $\liminf_{n\to +\infty} S_{x_n}\ge
S(p(\cdot),q(\cdot),\R^N\backslash B_R)$. We obtain the result
taking the limit $R\to +\infty$.
\end{proof}

Assume that the infimum in the right hand side of \eqref{CCPCond} is
attained at a point $x_0\in\A$. According to the result in
\cite{FBSS2}, the condition \eqref{CCPCond} will hold for some $v$
if $p$ and $q$ are $C^2$ in a neighborhood of $x_0$, $p$ (resp. $q$)
has a local minimum (resp. maximum) at $x_0$ and at least one of the
hessian matrix $D^2p(x_0)$ or $D^2q(x_0)$ is non-zero. We will now
see that the same kind of result holds if the infimum in the right hand side of \eqref{CCPCond} is attained at infinity. 
To do this we need the following definition

\begin{defi}\label{DefiC2Infinity}
 We say that a function $f:\R^N\to \R$ has a Taylor expansion of order 2 at $\infty$, 
or equivalently that $f$ is $C^2$ at infinity, 
if there exists $f(\infty),a_1,..,a_N\in \R$ and  $a_{i,j}\in \R$, $i,j=1,..,N$, such that
\begin{eqnarray}\label{Taylorp}
 f(x) = f(\infty) +\sum_{i=1}^N a_i\frac{x_i}{|x|^2} + \frac{1}{2}\sum_{i,j=1}^N a_{i,j}\frac{x_ix_j}{|x|^4} + o\Big(\frac{1}{|x|^2}\Big),\qquad |x|\gg 1.
\end{eqnarray}
We  say that $a_1,..,a_N$ are the first order partial derivatives of $f$ at $\infty$ and that the $a_{i,j}$ are the second-order partial derivatives of $f$ at $\infty$. \\
We  denote $\p_if(\infty):=a_i$, $\nabla f(\infty):=(\p_1f(\infty),...,\p_N f(\infty))$ the gradient of $f$ at $\infty$,
$\p_{ij}f(\infty):=a_{i,j}$
and $D^2f(\infty):=(\p_{ij}f(\infty))_{i,j=1,..,N}\in \R^{N\times N}$ the Hessian matrix of f at $\infty$.
\end{defi}

\noindent Notice that the usual algebraic rules of computations of partial derivatives apply for the partial derivatives at $\infty$.

This definition is justified by the following argument.
Denoting by $\Phi_N:\mathbb{S}^N-\{P\}\to \R$ the stereographic projection from the north pole $P$, 
we can define $\tilde f:\mathbb{S}^N\to \R$ by
$\tilde f=f\circ\Phi_N$. Then, if there exists $f(\infty):=\lim_{|x|\to\infty}f(x)$, $\tilde f$ is continuous on $\mathbb{S}^N$, and studying $f$ near $\infty$ is equivalent to study $\tilde f$ near $P$.
To do so we consider the chart $(\mathbb{S}^N-\{S\},\Phi_S)$ around $P$, where $\Phi_S$ is the stereographic projection from the south pole $S$. Then, in this chart, assuming $\tilde f$ of class $C^2$, we can write a Taylor expansion of 
$\tilde f$ at $0=\Phi_S(N)$:
\begin{eqnarray}\label{tildef}
 \tilde f(x)=a_0+\sum_{i=1}^N a_ix_i+\frac{1}{2}\sum_{i,j=1}^N  a_{ij}x_ix_j+o(|x|^2),\qquad x\in\mathbb{S}^N \text{ close to } P,
\end{eqnarray}
where $x_i$ is the i-th component of $\Phi_S(x)$.
Notice that $a_i=\p_i\tilde f(0)$ and $a_{ij}=\p_{ij}\tilde f(0)$.
Then \eqref{tildef} can be rewritten as
$$ f(x)=a_0+\sum_{i=1}^N a_i(\Phi_S(\Phi_N^{-1}(x)))_i+\frac{1}{2}\sum_{i,j=1}^Na_{ij}(\Phi_S(\Phi_N^{-1}(x)))_i(\Phi_S(\Phi_N^{-1}(x)))_j+o(|\Phi_S(\Phi_N^{-1}(x))|^2) $$
for $x\in \R^N,\,|x|\gg 1$.
Since $\Phi_S(\Phi_N^{-1}(x))=\frac{x}{|x|^2}$, we see that \eqref{tildef} is equivalent to \eqref{Taylorp}.

Assume that $f$ has a limit at $\infty$ in the sense that there
exists $f(\infty):=\lim_{|x|\to\infty}f(x)$. We say that $f$ has a
local minimum (res[. maximum) at $\infty$ if $f(x)\ge f(\infty)$
(resp. $f(x)\le f(\infty)$) for $|x|\gg 1$. If this is the case and
if moreover $f$ has a Taylor expansion of order 2, then $\nabla
f(\infty)=0$ and the hessian matrix $D^2f(\infty)$ is nonnegative
(resp. non-positive).

We can now state our existence result for equation \eqref{MainEqu}.

\begin{teo}\label{CSexistenciaInfinity}
Assume that assumptions (H1)-(H7) holds and also that
\begin{itemize}
\item $p(\infty)<\sqrt{N}$, 
\item $p$ (resp. $q$) has a local minimum (resp. maximum) at every point of $\A\cup\infty$,
\item  $p$ and $q$ are $C^2$ at any point of $\A\cup\infty$ and  for any $x_0\in \A\cup\infty$,
 at least one of the  hessian matrices $D^2p(x_0)$ or $D^2p(x_0)$ is non-zero.
\end{itemize}
Then \eqref{CCPCond} holds. In particular equation \eqref{MainEqu} has a non-trivial solution.
\end{teo}

\begin{proof}
As we said above we only need to verify that \eqref{CCPCond} holds
when the infimum in the right hand side of \eqref{CCPCond} is
attained at $\infty$. 

In view of our assumption we can apply
Theorem. \ref{PropIgaldad} to obtain
$$ S_\infty = K(N,p(\infty))^{-1}. $$
It is well-known that the extremals for $K(N,p(\infty))^{-1}$ are
the of the form $U_{\lambda, x_0}(x) :=
\lambda^{-\frac{n-p(\infty)}{p(\infty)}} U(\tfrac{x-x_0}{\lambda})$,
$x_0\in\R^n$, $\lambda>0$, where
$U(x)=(1+|x|^\frac{p}{p-1})^\frac{p-n}{p}$. 
Given some direction $\nu\in\R^N$, $|\nu|=1$, that will be specified later, 
and positive real numbers $r_\eps$ satisfying \eqref{HypR}, 
consider the function $u_\eps$ defined by \eqref{DefTestFunction} in Appendix B. 
Letting $C:=K(n,p(\infty))^{-\frac{n}{q(\infty)}}\|U\|_{q(\infty)}^{-1}$, it is easilt seen that 
$W:=CU$ satisfies $-\Delta_p W=W^{q(\infty)-1}$. We then consider
$w_\eps(x) := C u_\eps(x)$ as test-function in place of $v$ in \eqref{CCPCond}. 
In view of the definition of $\F$ in \eqref{DefF}
we then have
$$ J_\eps(t):=\F(tw_\eps) 
=\int_{\R^N} g(x,t)|\nabla u_\eps|^{p(x)}\,dx + \int_{\R^N} h(x,t) |u_\eps|^{p(x)}\,dx - \int_{\R^N} f(x,t) |u_\eps|^{q(x)}\,dx  $$
with $g(x,t)=\frac{(tC)^{p(x)}}{p(x)}$, $f(x,t)= \frac{(tC)^{q(x)}K(x)}{q(x)}$, $h(x,t)=\frac{(tC)^{p(x)}k(x)}{p(x)}$.
For ease of notation we let $p:=p(\infty)$ and $q:=q(\infty)$.
Consider 
$$ J_0(t):=\frac{(tC)^p}{p} \|\nabla U\|_p^p  -  \frac{(tC)^q K(\infty)}{q} \|U\|_q^q
 = K(N,p)^{-n}\Big(\frac{t^p}{p}-\frac{t^q}{q}K(\infty)\Big) $$
and
\begin{eqnarray*}
 J_1(t) & := & \frac{n(tC)^qK(\infty)}{2q^2}A(q,\nu)\|U\|_q^q - \frac{n(tC)^p}{2p^2}A(p,\nu)\|\nabla U\|_p^p \\
& = & \frac{n}{2}K(N,p)^{-n}\Big(\frac{t^q}{q^2}K(\infty) A(q,\nu) - \frac{t^p}{p^2}A(p,\nu) \Big).
\end{eqnarray*}
where  
$$A(p,\nu):= \sum_{i,j=1}^n \p_{i,j}p(\infty) \nu_i\nu_j = (D^2p(\infty)\nu,\nu)$$ 
and 
$$A(q,\nu):= \sum_{i,j=1}^n \p_{i,j}q(\infty) \nu_i\nu_j = (D^2q(\infty)\nu,\nu).$$ 
Using prop. \ref{PropAsymptq}, prop. \ref{PropAsymptGrad} and prop. \ref{PropAsymptp} in Appendix B, 
we then obtain that
\begin{equation}\label{AsimptotF}
 J_\eps(t) =  J_0(t) + J_1(t)\frac{\ln\,\eps}{|x_\eps|^2} + o\Big( \frac{\ln\,\eps}{|x_\eps|^2} \Big)
\end{equation}
$C^1$ uniformly in $t$ for $t$ in a compact subset of $[0,+\infty)$. 
Observe that $J_0$ attains its maximum at $t_0:=K(\infty)^\frac{1}{p-q}$ with $J_0(t_0)= \frac1n K(n,p)^{-n} K(\infty)^{-\frac{n-p}{p}}$.
Moreover this is a non-degenerate maximum since
$J_0''(t_0)=(p-q)K(n,p)^{-n}K(\infty)^\frac{2-p}{q-p}\neq 0$. It follows that $J_\eps$ attains its maximum
at $t_\eps = t_0+a\frac{\ln\,\eps}{|x_\eps|^2} + o\Big( \frac{\ln\,\eps}{|x_\eps|^2} \Big) $
with $a=-\frac{J_1'(t_0)}{J_0''(t_0)}$.
We thus obtain
\begin{eqnarray*}
 \sup_{t\ge 0} J_\eps(t) &=&  \sup_{t\ge 0} \F(w_\eps) = J_\eps(t_\eps)
 = J_0(t_0) + J_1(t_0) \frac{\ln\,\eps}{|x_\eps|^2} + o\Big( \frac{\ln\,\eps}{|x_\eps|^2} \Big)
\end{eqnarray*}
where 
$$J_1(t_0) = \frac{n}{2} K(n,p)^{-n} K(\infty)^\frac{p}{p-q} \Big( \frac{1}{q^2}A(q,\nu) - \frac{1}{p^2}A(p,\nu)\Big).$$ 
For  the sufficient existence condition \eqref{CCPCond} to hold, it thus suffices to take $v=w_\eps$ with $\eps$ small if the strict inequality
$J_1(t_0)<0$ holds. Since $q$ has a local maximum at infinity and $p$ a local minimum we already know that $J_1(t_0)\le 0$ for any direction $\nu$.
For the strict inequality to hold it thus suffices to assume that there exists a direction $\nu$ such that 
$A(p,\nu)>0$ or $A(q,\nu)<0$ i.e. to assume that
one of the hessian matrices $D^2p(\infty)$ or $D^2q(\infty)$ is non-zero.
\end{proof}

\section*{Appendix A: a compact embedding from $D^{1,p(\cdot)}(\R^N)$ into $L^{p(\cdot)}(\R^N,k\,dx)$. }

In this section, we give a sufficient condition for assumption (H6) to hold i.e. to have a compact embedding 
from  $D^{1,p(\cdot)}(\R^N)$, the closure of $C^\infty_c(\R^N)$ for the norm $\|\nabla u\|_{L^{p(\cdot)}(\R^N)}$, into
the weighted Lebesgue space $L^{p(\cdot)}(\R^N,k\,dx)$. Such a result was 
proved in the constant-exponent framework in \cite{Chabrowski}[lemma 1]. 
We adapt  his proof  to the variable exponent setting.

\begin{prop}\label{Chab}
Consider  a  measurable function $s:\R^N\to [1,+\infty]$ such that $s\ge p$ and $\inf_{\R^N} p^*-s >0$.  
Assume that $k \in L^1(\R^N)\cap L^\frac{s(\cdot)}{s(\cdot)-p(\cdot)}_{loc}(\R^N)$ is non-negative and satisfies 
$$ \lim_{|x|\to\infty}\|k\|_{L^\frac{s(\cdot)}{s(\cdot)-p(\cdot)}(Q(x,l))}=0 $$
for some $l>0$. Here $Q(x,l)$ denotes the cube centered at $x$ with sides of length $l$ parallel to the axis.
Then $D^{1,p(\cdot)}(\R^N)$ is compactly embedded into $L^{p(\cdot)}(\R^N,k\,dx)$.
\end{prop}

\noindent The proof follows closely the proof \cite{Chabrowski}[lemma 1] with the following useful trick due to P. Hasto \cite{Hasto}[Thm 2.4] in the final step.
Given a sequence $(q_j)_j$ in $(1,+\infty)$, consider the space $\ell^{(q_j)}$ of sequences $x=(x_j)_j$
such that $\rho(x):=\sum_j |x_j|^{q_j}<\infty$. This space is endowed with the norm
$ \|x\|_{\ell^{(q_j)} }:= \inf\,\{\lambda>0 \mbox{ s.t. } \rho(x/\lambda)\le 1\}$. A. Nekvinda \cite{Nekvinda}[Thm 4.3] proved that
if there exists $q(\infty)$ such that $|q_j-q(\infty)|\le C/\ln\,(e+j)$ then $\ell^{(q_j)}\approx \ell^{q(\infty)}$.
Using this result P. Hasto \cite{Hasto}[Thm 2.4] proved   that given a partition of $\R^N$ by  cubes $Q_l$, $l\ge 1$, of same length satisfying
\begin{equation}\label{CondCube}
 dist(Q_k,0)<dist(Q_l,0) \qquad \mbox{ if } k<l,
\end{equation}
and if $p$ is log-Holder continuous exponent such that
\begin{equation}\label{CondpHasto}
|p(x)-p(\infty)|\le C/\ln\,(e+|x|) \qquad \mbox{ for all } x\in\R^N,
\end{equation}
then
\begin{equation}\label{ResultadoHasto}
 \|g|_{L^{p(\cdot)}(\R^N)}\approx \| (\|g\|_{L^{p(\cdot)}(Q_j)})_j \|_{\ell^{(q_j)}}
\approx \| (\|g\|_{L^{p(\cdot)}(Q_j)})_j \|_{\ell^{p(\infty)}},
\end{equation}
where $q_j=p_{Q_j}^+$. Actually the first $\approx$ follows from the first part of the proof of \cite{Hasto}[Thm 2.4]. It is easily seen that we can take $q_j$ of the form $q_j=p(x_j)$ where $x_j\in Q_j$. In particular we will use this result with
$q_j=p_{Q_j}^-$.

\begin{proof}
Without loss of generality we may assume that $l=1$. We first verify that it is sufficient to show that for every $\delta>0$  there exists $j>0$ such that
\begin{equation}\label{comp}
\|f-f\chi_{Q(0,j)}\|_{L^{p(\cdot)}(\R^N,k\,dx)}<\delta
\end{equation}
for all $f\in D^{1,p(\cdot)}(\R^N)$ such that $\|\nabla f\|_{L^{p(\cdot)}(\R^N)}\leq1$.
Here $\chi_{Q(0,j)}$ denotes the characteristic function of the cube $Q(0,j)$.
In fact, assume that \eqref{comp} holds and take  a bounded sequence $(f_n)_n$ in $D^{1,p(\cdot)}(\R^N)$.
Since $s$ is assumed to be strictly subcritical, we can assume that there exists
$f\in D^{1,p(\cdot)}(\R^N)$ such that $f_m\to f$ in $L^{s(\cdot)}(Q(0,R))$ for any $R>0$ and
$\nabla f_m\rightharpoonup \nabla f$ in $L^{p(\cdot)}(\R^N)$.
It follows from \eqref{comp} that
$$ \|f_m-f\|_{L^{p(\cdot)}(\R^N-Q(0,j),k\,dx)}<2\delta.  $$
On the other hand  we have using H\"older inequality that
\begin{align*}
\int_{Q(0,j)}|f_m-f|^{p(\cdot)}k\,dx
&\leq   \||f_m-f|^{p(\cdot)}\|_{L^\frac{s(\cdot)}{p(\cdot)}(Q(0,j))}  \|k\|_{L^\frac{s(\cdot)}{s(\cdot)-p(\cdot)}(Q(0,j))}\\
&\leq  C\max\{\|f_m-f\|_{L^{s(\cdot)}(Q(0,j))}^{p^+};\|f_m-f\|_{L^{s(\cdot)}(Q(0,j))}^{p^-}\}
\end{align*}
which goes to $0$ as $m\to +\infty$  since $f_m\to f$ in $L^{s(\cdot)}(Q(0,j))$,
We can thus conclude that $f_m\to f$ in $L^{s(\cdot)}(\R^N)$.

\medskip

We now prove \eqref{comp}.
We fix a function $f\in D^{1,p(\cdot)}(\R^N)$ such that $\|\nabla f\|_{L^{p(\cdot)}(\R^N)}\leq1$,
and consider a partition of $\R^N$ by  cubes $Q_l$, $l\ge 1$, of unit side length satisfying \eqref{CondCube}.
Independently, given $\eta>0$ we  fix  $j\in\N$ such that
$$\|k\|_{L^1(\R^N-Q(0,j))}\leq\eta $$
and
$$ \|k\|_{L^\frac{s(\cdot)}{s(\cdot)-p(\cdot)}(Q_l)}<\eta  \qquad  \mbox{    for every $Q_l$ outside $Q(0,j)$.} $$
Denote $f_{Q_l}:=\int_{Q_l} f(x)\,dx$ the mean-value of $f$ on $Q_l$.
Then by Sobolev and Poincar\'e inequality there exists a constant $C$ independent of $f$ and $Q_l$ such that
$$ \|f-f_{Q_l}\|_{L^{s(\cdot)}(Q)}\leq C\|\nabla f\|_{L^{p(\cdot)}(Q_l)}  $$
and
$$ |f_{Q_l}|  \leq C \|f\|_{L^{p^*(\cdot)}(Q_l)}  \leq C \|f\|_{L^{p^*(\cdot)}(\R^N)}
                             \le C\|\nabla  f\|_{L^{p(\cdot)}(\R^N)} = C. $$
Let $p_l^+:=  \max_{Q_l}p$ and $p_l^-:=  \min_{Q_l}p$. It follows that for any cube $Q_l$ we have
\begin{eqnarray*}
\int_{Q_l}|f|^{p(\cdot)}k\,dx
& \le & 2^{p^+-1} \int_{Q_l}|f-f_{Q_l}|^{p(\cdot)}k\,dx + 2^{p^+-1} \int_{Q_l} |f_{Q_l}|^{p(\cdot)}k\,dx \\
&=& C\||f-f_{Q_l}|^{p(\cdot)}\|_{L^\frac{s(\cdot)}{p(\cdot)}(Q_l)}  \|k\|_{L^\frac{s(\cdot)}{s(\cdot)-p(\cdot)}(Q_l)}
+ C\int_{Q_l} k\,dx \\
&\leq & C\max\{\|\nabla f\|_{L^{p(\cdot)}(Q_l)}^{p_j^+};\|\nabla f\|_{L^{p(\cdot)}(Q_l)}^{p_j^-}\} \eta + C \int_{Q_l} k\,dx.
\end{eqnarray*}
Since $\|\nabla f\|_{L^{p(\cdot)}(\R^N)}\leq1$, we eventually obtain for any cube $Q_l$ that
$$ \int_{Q_l}|f|^{p(\cdot)}r\,dx  \le C\|\nabla f\|_{L^{p(\cdot)}(Q_l)}^{q_j^-} \eta + C \int_{Q_l} k\,dx. $$
Summing  these inequalities over all the cubes $Q_l$ outside $Q(0,j)$ using \eqref{ResultadoHasto} to handle the first term,
we obtain
$$ \int_{\R^N-Q(0,j)} |f|^{p(\cdot)}k\,dx \leq C\eta+ C\int_{\R^N-Q(0,j)} k\,dx \leq  C\eta $$
which is \eqref{comp}.
\end{proof}

\section*{Appendix B: test-functions computations at infinity. }

Let $U(x)=(1+|x|^\frac{p(\infty)}{p(\infty)-1})^{-\frac{n-p(\infty)}{p(\infty)}}$. It is well-known that $U$ is an extremal for 
$K(N,p(\infty))^{-1}$. 
Given  a cut-off function $\eta\in C^\infty_c([0,+\infty),[0,1])$ supported in $[0,2]$ such that $\eta\equiv 1$ in $[0,1]$, and points $x_\eps\in\R^n$ such that $|x_\eps|\to +\infty$, we consider the test function
\begin{equation}\label{DefTestFunction}
u_\eps(x)=\eps^{-\frac{n-p(\infty)}{p(\infty)}}U\Big(\frac{x-x_\eps}{\eps}\Big)\eta(|x-x_\eps|), \qquad \eps\ll 1.
\end{equation}
The points $x_\eps$ are chosen such that  $x_\eps:=r_\eps\nu$ where $\nu\in \R^N$, $|\nu|=1$, and $r_\eps$ satisfies 
\begin{eqnarray}\label{HypR}
\eps^{-\frac{p(\infty)-1}{4}},\eps^{-1/2}, \eps^{-\frac12\frac{n-p(\infty)}{p(\infty)-1}+\eta}\gg r_\eps \gg |\ln\eps|,
\end{eqnarray}
for some $\eta>0$. 

The result is the following:

\begin{prop}\label{PropAsymptq}
Consider $f\in C(\R^N)$. 
 Assume that $q$ and $f$ have a Taylor expansion at $\infty$ of order 2 and that $\nabla f(\infty)=\nabla q(\infty)=0$.
Then
\begin{eqnarray}\label{Asymptq}
 \int_{\R^N} f|u_\eps|^{q(x)}\,dx = A_0 + \frac{\ln\eps}{|x_\eps|^2}A_1 + \frac{1}{|x_\eps|^2}A_2 + o\Big(\frac{1}{|x_\eps|^2}\Big)
\end{eqnarray}
where, denoting $A(q,\nu):= \sum_{i,j=1}^n \p_{i,j}q(\infty) \nu_i\nu_j = (D^2q(\infty)\nu,\nu)$
and $A(f,\nu):= \sum_{i,j=1}^n \p_{i,j}f(\infty) \nu_i\nu_j$,
$$ A_0 = f(\infty)\int_{\R^N} U^{q(\infty)}(x)\,dx, \qquad
 A_1 = - \frac{n f(\infty) }{2q(\infty)} A(q,\nu) \int_{\R^N} U^{q(\infty)}(x)\,dx, $$
 $$ A_2 =\frac{f(\infty)}{2} A(q,\nu)  \int_{\R^N} U(x)^{q(\infty)}\ln\,U(x)\,dx + \frac12 A(f,\nu)  \int_{\R^N} U(x)^{q(\infty)}\,dx. $$
\end{prop}

\begin{proof}
First
\begin{eqnarray}\label{Asympt1}
 \int_{\R^N} f|u_\eps|^{q(x)}\,dx =
\int_{B(1/\eps)} f(x_\eps+\eps x) \eps^{n-\frac{n}{q(\infty)}q(x_\eps+\eps x)} U(x)^{q(x_\eps+\eps x)}\,dx + R_\eps
\end{eqnarray}
with
$$ R_\eps
= \int_{\frac{1}{\eps}\le |x|\le \frac{2}{\eps}} f(x_\eps+\eps x) \eps^{n-\frac{n}{q(\infty)}q(x_\eps+\eps x)} U(x)^{q(x_\eps+\eps x)}
\eta(\eps x)^{q(x_\eps+\eps x)}\,dx. $$
Fix some positive $\gamma \le p(\infty)/(N(2p(\infty)-1))$. 
Recalling that  $\lim_{|x|\to +\infty}q(x)=q(\infty)$, $|x_\eps|\to +\infty$, and $|x_\eps+\eps x|\ge |x_\eps|-2$ for $|x|\le 2/\eps$,
we have, for $\eps$ small enough and any $|x|\le 2/\eps$, that $q(\infty)(1-\gamma)\le q(x_\eps+\eps x)\le q(\infty)(1+\gamma)$.
It follows that for $\eps$ small,
\begin{eqnarray*}
|R_\eps|
&\le & (|f(\infty)|+1) \eps^{-N\gamma} \int_{|x|\ge \frac{\delta}{\eps}} U(x)^{q(\infty)(1-\gamma)}
= \omega_{N-1}(|f(\infty)|+1) \eps^{-N\gamma}  \int_{\delta/\eps}^{+\infty} U(r)^{q(\infty)(1-\gamma)}  \\
&\le& C \eps^{-N\gamma + \frac{p(\infty)}{p(\infty)-1}N(1-\gamma)-N}
       = C \eps^{\frac{N}{p(\infty)-1} - \frac{N\gamma (2p(\infty)-1)}{p(\infty)-1}} 
			\le C \eps^{\frac{N-p(\infty)}{p(\infty)-1} }, 
\end{eqnarray*}
where $C$ is independent of $\gamma$. 
In view of \eqref{HypR}, we obtain that $R_\eps=o(|x_\eps|^{-2})$.

We now estimate the first term in the right hand side of
\eqref{Asympt1}. All the expansion will be uniform in $|x|\le
2/\eps$. First, for $|x|\le 1/\eps$, $|x_\eps+\eps
x|^2=|x_\eps|\Big(1+O\Big(\frac{1}{|x_\eps|}\Big)\Big)$ so that
$$ |x_\eps+\eps x|^{-4} = |x_\eps|^{-4}\Big(1+ O\Big(\frac{1}{|x_\eps|}\Big) \Big).  $$
In view of our assumptions we can write that
\begin{eqnarray*}
 q(x) = q(\infty) + \frac{1}{2}\sum_{i,j=1}^n \p_{i,j}q(\infty)\frac{x_ix_j}{|x|^4} + o\Big(\frac{1}{|x|^2}\Big),\qquad |x|\gg 1.
\end{eqnarray*}
We thus have uniformly in $|x|\le 1/\eps$ that
\begin{eqnarray}\label{Asymptq3}
 q(x_\eps+\eps x)
& = & q(\infty) +  \frac{1}{2}\p_{i,j}q(\infty)  (x_\eps^i+\eps x^i)(x_\eps^j+\eps x^j)|x_\eps|^{-4}\Big(1+ O\Big(\frac{1}{|x_\eps|}\Big) \Big)
+ o\Big(\frac{1}{|x|^2}\Big) \nonumber \\
& = & q(\infty) +  \frac{1}{2}\p_{i,j}q(\infty) \frac{x_\eps^i x_\eps^j}{|x_\eps|^4} + o\Big(\frac{1}{|x_\eps|^2}\Big). 
\end{eqnarray}
We can expand $f(x_\eps+\eps x)$ in the same way as
$$ f(x_\eps+\eps x)  =  f(\infty) +  \frac{1}{2}\p_{i,j}f(\infty) \frac{x_\eps^i x_\eps^j}{|x_\eps|^4} + o\Big(\frac{1}{|x_\eps|^2}\Big) $$
Moreover
\begin{eqnarray*}
\eps^{n-\frac{n}{q(\infty)}q(x_\eps+\eps x)}
&= & \exp\Big\{\Big( n-\frac{n}{q(\infty)}q(x_\eps+\eps x) \Big)\ln\eps\Big\}  \\
&=& 1 - \frac{n}{2q(\infty)} \p_{i,j}q(\infty)     \frac{x_\eps^i x_\eps^j}{|x_\eps|^2}\frac{\ln\,\eps}{|x_\eps|^2}
+ O\Big(\frac{(\ln\,\eps)^2}{|x_\eps|^4}\Big),
\end{eqnarray*}
and, recalling \eqref{HypR},
\begin{eqnarray*}
 U(x)^{q(x_\eps+\eps x)}
&=&U(x)^{q(\infty)}\exp\Big\{(q(x_\eps+\eps x)-q(\infty))\ln\,U(x) \Big\} \\
&=& U(x)^{q(\infty)}
\Big\{ 1 + \frac{1}{|x_\eps|^2}\frac{\ln\,U(x)}{2}\p_{i,j}q(\infty)  \frac{x_\eps^ix_\eps^j}{|x_\eps|^2} \Big)
 + o\Big(\frac{1}{|x_\eps|^2}\Big) \ln\,U(x).
\end{eqnarray*}
Noticing that $\lim_{\eps\to 0}\int_{\R^N\backslash B(1/\eps)} U(x)^{q(\infty)}(1+|\ln\,U(x)|)\,dx = 0$,
we deduce \eqref{Asymptq} by plugging these expansions in \eqref{Asympt1}.

\end{proof}

We now compute the asymptotic expansion of a term of the form $\int_{\R^N} g(x)|\nabla u_\eps|^{p(x)}\,dx$. 
We make the same hypothesis on $g$, $p$ than on $f$, $q$ before. The computations are very similar.

\begin{prop}\label{PropAsymptGrad}
  Assume that $p$ and $g$ have a Taylor expansion of order 2 at $\infty$ with $\nabla p(\infty)=\nabla g(\infty)=0$ and that $p(\infty)<\sqrt{N}$.
Then
\begin{eqnarray*}
 \int_{\R^N} g|\nabla u_\eps|^{p(x)}\,dx = B_0 + \frac{\ln\eps}{|x_\eps|^2}B_1 + \frac{1}{|x_\eps|^2}B_2 + o\Big(\frac{1}{|x_\eps|^2}\Big),
\end{eqnarray*}
where, denoting $A(p,\nu):= \sum_{i,j=1}^n \p_{i,j}p(\infty)$ and  $A(g,\nu):= \sum_{i,j=1}^n \p_{i,j}g(\infty)$,
$$ B_0 = g(\infty)\int_{\R^N} |\nabla U|^{p(\infty)}\,dx, \qquad  B_1 = - \frac{n g(\infty) }{2p(\infty)} A(p,\nu) \int_{\R^N} |\nabla U|^{p(\infty)}\,dx $$
$$  B_2 =\frac{g(\infty)}{2} A(p,\nu)  \int_{\R^N} |\nabla U|^{p(\infty)}\ln\,|\nabla U|\,dx + \frac12 A(g,\nu)  \int_{\R^N} |\nabla U|^{p(\infty)}\,dx. $$
\end{prop}

\begin{proof}
Denoting $u_\eps = U_\eps\eta_\eps$, with $U_\eps(x)=\eps^{-\frac{n-p(\infty)}{p(\infty)}}U\Big(\frac{x-x_\eps}{\eps}\Big)$,
and $\eta_\eps(x):=\eta(|x-x_\eps|)$,  we  first write that
$$ \int_{\R^N} g(x)|\nabla u_\eps|^{p(x)}\,dx
= \int_{\R^N} g(x)\Big|\eta_\eps\nabla U_\eps + U_\eps \nabla \eta_\eps\Big|^{p(x)}\,dx
= \int_{\R^N} g(x)|\nabla U_\eps|^{p(x)}\eta_\eps^{p(x)}\,dx + R_\eps, $$
where, using the well-known inequality
$$ \Big| |a+b|^p-|a|^p\Big| \le C\Big(|b|^p+|a|^{p-1}|b|+|a||b|^{p-1}| \Big),\qquad a,b\in\R^N, $$
with $C>0$ uniform in $p$ for $p$ in a compact subset of $(1,+\infty)$, we can estimate the error term $R_\eps$ by
\begin{eqnarray*}
&& |R_\eps| \le  (|g(\infty)|+1)C \int_{B_{x_\eps}(2)\backslash B_{x_\eps}(1)} U_\eps^{p(x)} + |\nabla U_\eps|^{p(x)-1}U_\eps + U_\eps^{p(x)-1}|\nabla U_\eps|\,dx \\
&=&  C\int_{B(2/\eps)\backslash B(1/\eps)} \eps^{N\Big(1-\frac{p(x_\eps+\eps x)}{p(\infty)}\Big)} 
\Big\{ \eps^{p(x_\eps+\eps x)} U^{p(x_\eps+\eps x)}
 + \eps |\nabla U|^{p(x_\eps+\eps x)-1}U \\ 
 && \hspace{5.5cm} + \eps^{p(x_\eps+\eps x)-1} |\nabla U|U^{p(x_\eps+\eps x)-1} \Big\}\,dx. 
\end{eqnarray*}
In view of \eqref{Asymptq3} with $p$ in place of $q$, we have $1-\frac{p(x_\eps+\eps x)}{p(\infty)}=O\Big(\frac{1}{|x_\eps|^2}\Big)$ uniformly
for $|x|\le 2/\eps$, with $|x_\eps|\gg\ln\eps$.
It follows that $\eps^{N\Big(1-\frac{p(x_\eps+\eps x)}{p(\infty)}\Big)}\to 1$ uniformly for $|x|\le 2/\eps$.
We then deduce that for $\eps$ small
\begin{eqnarray*}
|R_\eps|
&\le &  C\int_{\R^N\backslash B(1/\eps)}   \eps^{p(x_\eps+\eps x)} U^{p(x_\eps+\eps x)} + \eps|\nabla U|^{p(x_\eps+\eps x)-1}U  +
\eps^{p(x_\eps+\eps x)-1} |\nabla U|U^{p(x_\eps+\eps x)-1} \,dx \\
&\le & C\eps^{\min\,\{1,\frac12 (p(\infty)-1) \}}
\int_{\R^N\backslash B(1/\eps)}  U^{p(x_\eps+\eps x)} + |\nabla U|^{p(x_\eps+\eps x)-1}U  + |\nabla U|U^{p(x_\eps+\eps x)-1} \,dx.
\end{eqnarray*}
Using the definition of $U$ and writing $p(x_\eps+\eps x)\ge p(\infty)(1-\gamma)$ as before, it is easily seen that these integrals are finite if
$p(\infty)<\min\{\sqrt{N},\frac{N+1}{2}\}=\sqrt{N}$.
We thus obtain in that case that $R_\eps=o(\eps^{\min\,\{1,\frac12 (p(\infty)-1) \}})= o(1/|x_\eps|^2)$ by \eqref{HypR}.

Coming back to the beginning of the proof we thus have
\begin{eqnarray*}
  \int_{\R^N} g(x)|\nabla u_\eps|^{p(x)}\,dx
    & = & \int_{\R^N} g(x)|\nabla U_\eps|^{p(x)}\eta_\eps^{p(x)}\,dx + o\Big(\frac{1}{|x_\eps|^2}\Big) \\
    & = & \int_{B_{x_\eps}(1)} g(x)|\nabla U_\eps|^{p(x)}\,dx + \tilde R_\eps + o\Big(\frac{1}{|x_\eps|^2}\Big),
\end{eqnarray*}
where,  writing $p(x_\eps+\eps x)\ge p(\infty)(1-\gamma)$ as before,
\begin{eqnarray*}
 \tilde R_\eps
& \le & \int_{B_{x_\eps}(2)\backslash B_{x_\eps}(1)} g(x)|\nabla U_\eps|^{p(x)}\,dx
= \int_{B(2/\eps)\backslash B(1/\eps)} g(x_\eps+\eps x) \eps^{N\Big(1-\frac{p(x_\eps+\eps x)}{p(\infty)}\Big)} |\nabla U|^{p(x_\eps+\eps x)}\,dx \\
&\le&  C \int_{\R^N\backslash B(1/\eps)} |\nabla U|^{p(\infty)(1-\gamma)}\,dx \\
& \le & C\eps^{\frac{N-p(\infty)}{p(\infty)-1}-\eta}
\end{eqnarray*}
where $\eta=\gamma p(\infty)(N-1)/(p(\infty)-1)$ and 
the constant $C$ can be chosen independent of $\gamma$. 
By \eqref{HypR}, we obtain $\tilde R_\eps  = o(1/|x_\eps|^2))$.  
We thus have
\begin{eqnarray*}
 \int_{\R^N} g(x)|\nabla u_\eps|^{p(x)}\,dx &=& \int_{B(1/\eps)} g(x)|\nabla U_\eps|^{p(x)}\,dx + o\Big(\frac{1}{|x_\eps|^2}\Big)\\
&=& \int_{B(1/\eps)} g(x_\eps+\eps x)\eps^{n-\frac{n}{p(\infty)}p(x_\eps+\eps x)} |\nabla U(x)|^{p(x_\eps+\eps x}\,dx + o\Big(\frac{1}{|x_\eps|^2}\Big).
\end{eqnarray*}
Notice that the integral in the right hand side is exactly the
integral in the right hand side of \eqref{Asympt1} with $g,p,|\nabla
U|$ in place of $f,q,U$. We thus have the same asymptotic expansion
as in prop. \ref{PropAsymptq}.
\end{proof}

We eventually compute the asymptotic expansion of an expression like $\int_{\R^N} h(x)|u_\eps(x)|^{p(x)}\,dx$:

\begin{prop}\label{PropAsymptp}
 Assume that the function $h:\R^N\to \R$ is bounded.  
Then
$$ \int_{\R^N} h(x)|u_\eps(x)|^{p(x)}\,dx = O(\eps^{p(\infty)}) = o(|x_\eps|^{-2}). $$
\end{prop}

\begin{proof} First notice that since $p$ satisfies (H1), 
$$ \eps^{N-\frac{N}{q(\infty)}p(x_\eps+\eps x)}
=\eps^{p(\infty)}\eps^{-\frac{N}{q(\infty)}(p(x_\eps+\eps x)-p(\infty))}
= \eps^{p(\infty)} (1+o(1)) $$
uniformly for $\eps |x|\le 2$.  
We can then write that
\begin{eqnarray}\label{Asympt3}
 \int_{\R^N} h(x)|u_\eps(x)|^{p(x)}\,dx
&=& \int_{B(2/\eps)} h(x_\eps+\eps x) \eps^{N-\frac{N-p(\infty)}{p(\infty)}p(x_\eps+\eps x)}U(x)^{p(x_\eps+\eps x)} \eta(\eps |x|)^{p(x_\eps+\eps x)}\,dx \nonumber \\
&\le & C\eps^{p(\infty)} (1+o(1)) \int_{B(2/\eps)}  U(x)^{p(x_\eps+\eps x)} \,dx 
\end{eqnarray}
with $ |U(x)^{p(x_\eps+\eps x)}1_{B(2/\eps)}| 
\le U(x)^{p(\infty)(1+\gamma)}1_{\{U>1\}} + U(x)^{p(\infty)(1-\gamma)}1_{\{U\le 1\}})$ which is
integrable for $\gamma>0$ small since we assumed $p(\infty)<\sqrt{N}$. We conclude the proof noticing that $\eps^{p(\infty)}=o(|x_\eps|^{-2})$ 
by \eqref{HypR}. 
\end{proof}

\section*{Acknowledgements}

This work was partially supported by Universidad de Buenos Aires under grant UBACYT 20020100100400 and by CONICET (Argentina) PIP 5478/1438.

\def\ocirc#1{\ifmmode\setbox0=\hbox{$#1$}\dimen0=\ht0 \advance\dimen0
  by1pt\rlap{\hbox to\wd0{\hss\raise\dimen0
  \hbox{\hskip.2em$\scriptscriptstyle\circ$}\hss}}#1\else {\accent"17 #1}\fi}
  \def\ocirc#1{\ifmmode\setbox0=\hbox{$#1$}\dimen0=\ht0 \advance\dimen0
  by1pt\rlap{\hbox to\wd0{\hss\raise\dimen0
  \hbox{\hskip.2em$\scriptscriptstyle\circ$}\hss}}#1\else {\accent"17 #1}\fi}
\providecommand{\bysame}{\leavevmode\hbox to3em{\hrulefill}\thinspace}
\providecommand{\MR}{\relax\ifhmode\unskip\space\fi MR }
\providecommand{\MRhref}[2]{%
  \href{http://www.ams.org/mathscinet-getitem?mr=#1}{#2}
}
\providecommand{\href}[2]{#2}

\bibliographystyle{plain}
\bibliography{biblio}

\end{document}